\documentclass[11pt]{article}
\usepackage{amsmath,fullpage,amssymb,amsthm,hyperref,enumerate,shuffle}
\usepackage{tikz,color}
\usepackage{ifthen}
\usetikzlibrary{patterns,shapes,arrows,decorations.pathreplacing}
\tikzstyle{pnt}=[draw,ellipse,fill,inner sep=1pt]

\newtheorem{theorem}{Theorem}[section]

\newtheorem{lemma}[theorem]{Lemma}
\newtheorem{corollary}[theorem]{Corollary}
  
\newtheorem{problem}[theorem]{Problem}

\theoremstyle{definition}
\newtheorem{example}[theorem]{Example}  
\newtheorem{definition}[theorem]{Definition}   

\theoremstyle{remark}
\newtheorem{remark}[theorem]{Remark}

\newcommand{\Q}{\mathcal{Q}}
\newcommand{\QQ}{\overline{\mathcal{Q}}}
\newcommand{\C}{\mathcal{C}} 
\newcommand{\A}{\mathcal{A}} 
\newcommand{\B}{\mathcal{B}} 

\newcommand{\D}{\mathcal{D}}
\renewcommand{\S}{\mathcal{S}}
\newcommand{\la}{\lambda}
\newcommand{\si}{\sigma}

\newcommand{\nI}{\overline{I}}
\DeclareMathOperator\des{des}
\DeclareMathOperator\Des{Des}
\DeclareMathOperator\plat{plat}
\DeclareMathOperator\wdes{wdes}
\DeclareMathOperator\pea{pea}
\DeclareMathOperator\lpea{lpea}
\DeclareMathOperator\hpea{hpea}
\newcommand\LK{\hat{\kappa}}
\newcommand\LKC{\kappa}
\DeclareMathOperator\mat{mat}

\newcommand\id{\iota}

\newcommand{\poly}{C}
\newcommand{\Cat}{\mathrm{Cat}}
\newcommand{\card}[1]{{\lvert #1 \rvert}}
\newcommand{\delete}[1]{}
\newcommand\nn{[n]\sqcup[n]}
\newcommand\nnk{\bigsqcup^k[n]}
\newcommand\hf{\hat{f}}
\newcommand\hfk{\hat{f}_k}
\newcommand\hg{\hat{g}}
\newcommand\hphi{\hat{\phi}}
\newcommand\N{\texttt{n}}
\newcommand\E{\texttt{e}}
\newcommand\ba{\textcolor{red}{|}}
\newcommand\ol{\overline}
\newcommand\s{s}
\newcommand\dy{d}

\newcommand\grid[1]{
 \draw (0,0) -- (#1,0) -- (#1,#1) -- (0,0);
 \foreach \x in {1,...,#1} {
        \draw[thin,dotted] (\x,0) -- (\x,\x) -- (#1,\x);
      }
}

\newcommand\doublegrid[1]{
 \draw[thin,dotted] (0,0) grid (#1,#1);
 \draw (0,0) rectangle (#1,#1);
 \draw (0,0)--(#1,#1); 
}

\newcommand\labelgrid[1]{
	\foreach [evaluate={\x =\sig[\i]}] \i in {1,...,#1} {
    	\draw (\i-.5,0) node[below] {\x};
		\draw (#1,\i-.5) node[right] {\x};
	}
}

\newcommand\labelgridDes[1]{
    \labelgrid{#1}
    \foreach [evaluate={\x =\sig[\i]},evaluate={\xx =\sig[\i-1]}] \i in {1,...,#1} {
    	\ifthenelse{\xx>\x} {\draw[red] (\i-1,-.2)--(\i-1,-.8); 
    	                     \draw[red] (#1+.2,\i-1)--(#1+.8,\i-1);}
	}
}

\newcommand\Desdots[1]{
    \foreach \i in #1 {
        \draw[thick,red] (\i-.3,\i)--(\i+.3,\i);
        \draw[thick,red] (\i,\i-.3)--(\i,\i+.3);
        }
}

\newcommand\decorategrid[1]{
	\foreach [evaluate={\x =\sig[\i]},evaluate={\xx =\sig[\i-1]}] \i in {2,...,#1} {
    	\ifthenelse{\x<\xx}{\draw[red] (\i-1,0) -- (\i-1,\i-1) -- (#1,\i-1);}{};
	}
	\foreach [evaluate={\x =\sig[\i]},evaluate={\ii=\i-1}] \i in {2,...,#1} {
		\foreach [evaluate={\y =\sig[\j]}] \j in {1,...,\ii} {
			\ifthenelse{\x<\y}{\draw[red] (\i-1,\j) -- ++(.3,-.3);}{\draw[red] (\i,\j-1) -- ++(-.3,.3);};
		}
	}
}

\def\e{-- ++(1,0)}
\def\n{-- ++(0,1)}

\newcommand\gle[3]{
\draw (#1,#2) node[scale=.8] {#3};
}

\newcommand\ord[3]{
\draw[darkgray] (#1-.5,#2-.5) node[scale=.8] {{\texttt #3}};
}

\newcommand\arc[2]{ \draw[blue,very thick] (#1)  to [bend left=45] (#2);}
\newcommand\rect{\, \begin{tikzpicture}[scale=.25]
      \draw (0,0) rectangle (2,1);
\end{tikzpicture}\, 
}

\title{Descents on nonnesting multipermutations}

\author{Sergi Elizalde\thanks{Department of Mathematics, Dartmouth College, Hanover, NH 03755. \texttt{sergi.elizalde@dartmouth.edu}}}

\date{}

\begin{document}

\maketitle

\begin{abstract}
    Motivated by recent results on quasi-Stirling permutations, which are permutations of the multiset $\{1,1,2,2,\dots,n,n\}$ that avoid the ``crossing'' patterns $1212$ and $2121$, we consider nonnesting permutations, defined as those that avoid the patterns $1221$ and $2112$ instead. We show that the polynomial giving the distribution of the number of descents on nonnesting permutations is a product of an Eulerian polynomial and a Narayana polynomial. It follows that, rather unexpectedly, this polynomial is palindromic. We provide bijective proofs of these facts by composing various transformations on Dyck paths, including the Lalanne--Kreweras involution.

\end{abstract}


\section{Introduction}

\subsection{Permutations of multisets and descents}

In a seminal 1978 paper~\cite{gessel_stirling_1978}, Gessel and Stanley introduced {\em Stirling permutations}, which are permutations $\pi_1\pi_2\dots\pi_{2n}$ of the multiset $\{1,1,2,2,\dots,n,n\}$ satisfying that, if $i<j<k$ and $\pi_i=\pi_k$, then $\pi_j>\pi_i$. 
A restatement of this requirement is that $\pi$ has no subsequence $\pi_{i}\pi_{j}\pi_{k}$ whose entries are in the same relative order as $212$. More generally, given two sequences $\pi=\pi_1\pi_2\dots\pi_m$ and $\sigma=\sigma_1\sigma_2\dots\sigma_r$ of positive integers, we say that $\pi$ {\em avoids} $\sigma$ if there is no subsequence $\pi_{i_1}\pi_{i_2}\ldots\pi_{i_r}$ (where $i_1<i_2<\dots<i_r$) whose entries are in the same relative order as those of $\sigma$.
We will use the notation $[n]=\{1,2,\dots,n\}$ and $\nn=\{1,1,2,2,\dots,n,n\}$.

The name Stirling permutations comes from the property that their descent polynomials give a combinatorial interpretation of the numerators of certain generating functions for Stirling numbers. We say that $i$ is a {\em descent} of $\pi=\pi_1\pi_2\dots\pi_m$ if $\pi_i>\pi_{i+1}$, 
that it is a {\em plateau} if $\pi_i=\pi_{i+1}$, and
that it is a {\em weak descent} if $\pi_i\ge\pi_{i+1}$.
Denote the {\em descent set} of $\pi$ by 
$$\Des(\pi)=\{i\in[m-1]:\pi_i>\pi_{i+1}\},$$
its number of descents by $\des(\pi)=\card{\Des(\pi)}$, its number of plateaus by 
$$\plat(\pi)=\card{\{i\in[m-1]:\pi_i=\pi_{i+1}\}},$$ and its number
of weak descents by $\wdes(\pi)=\des(\pi)+\plat(\pi)$.
We require $i\in[m-1]$ in these definitions, following~\cite{foata_theorie_1970,graham_concrete_1994,stanley_enumerative_2012,archer_pattern_2019}, but we note that some other papers (e.g. \cite{gessel_stirling_1978,bona_real_2008,janson_generalized_2011,elizalde_descents_2021}) define the last position $m$ to always be a descent of $\pi$, resulting in one more descent than in our definition.

Let $\S_n$ denote set of permutations of $[n]$. The distribution of descents on $\S_n$ is given by the Eulerian polynomials\footnote{
The definition of $A_n(t)$ in~\cite{stanley_enumerative_2012} for $n\ge1$ has an extra factor of $t$.}
\begin{equation}\label{eq:Eulerian_def}
A_n(t)=\sum_{\sigma\in\S_n} t^{\des(\sigma)}.
\end{equation} 
It is well known \cite[Prop.\ 1.4.4 and 1.4.5]{stanley_enumerative_2012} that, for $n\ge1$,
\begin{equation}\label{eq:Eulerian}\sum_{m\ge0} m^n t^m=\frac{tA_n(t)}{(1-t)^{n+1}},\end{equation}
and that 
\begin{equation}\label{eq:EulerianGF}\sum_{n\ge0} A_n(t) \frac{z^n}{n!}=\frac{t-1}{t-e^{(t-1)z}}.\end{equation}
Note also that $A_n(t)$ is palindromic, 
in the sense that the coefficient of $t^r$ equals the coefficient of $t^{n-1-r}$ for all $r$; equivalently,
\begin{equation}\label{eq:Apali}
A_n(t)=t^{n-1}A_n(1/t).
\end{equation}
This is because 
\begin{equation}\label{eq:desRsi}
\des(\sigma^R)=n-1-\des(\sigma),
\end{equation}
where $\sigma^R=\sigma_n\sigma_{n-1}\dots\sigma_1$ denotes the reversal of $\sigma$.

Gessel and Stanley~\cite{gessel_stirling_1978} proved that, in analogy to Equation~\eqref{eq:Eulerian}, the descent polynomial for the set $\Q_n$ of Stirling permutations of length $2n$ satisfies
$$\sum_{m\ge0} S(m+n,m)\, t^m=\frac{t\,\sum_{\pi\in\Q_n} t^{\des(\pi)}}{(1-t)^{2n+1}},$$ 
where $S(k,m)$ denotes the Stirling numbers of the second kind. They also noted that  $|\Q_n|=(2n-1)\cdot(2n-3)\cdot\dots\cdot 3\cdot 1$.
Other interesting properties of the distribution of descents and plateaus on Stirling permutations, as well as generalizations to multisets containing a prescribed number of copies of each element in $[n]$, have been subsequently discovered \cite{brenti_unimodal_1989,bona_real_2008,janson_plane_2008,haglund_stable_2012, park_r-multipermutations_1994,park_inverse_1994,park_p-partitions_1994,janson_generalized_2011}. The term {\em multipermutation} is often used to refer to a permutation of a multiset.

In~\cite{archer_pattern_2019}, Archer et al. introduced a variation of Stirling permutations, which they call {\em quasi-Stirling permutations} and denote by $\QQ_n$. These are permutations $\pi_1\pi_2\dots\pi_{2n}$ of the multiset $\nn$ that avoid  $1212$ and $2121$, meaning that there do not exist $i<j<\ell<m$ such that $\pi_i=\pi_\ell$ and $\pi_j=\pi_m$ (and $\pi_i\neq\pi_j$, although this condition is superfluous since each element appears only twice). By definition, $\Q_n\subseteq\QQ_n$.
Archer et al.~\cite{archer_pattern_2019} showed that 
$$
|\QQ_n|=n!\,\Cat_n=\frac{(2n)!}{(n+1)!},\quad  \text{where }\Cat_n=\frac{1}{n+1}\binom{2n}{n}
$$
is the $n$th Catalan number, and counted such permutations with a given number of plateaus.

The generating function for quasi-Stirling permutations with respect to the number of descents and plateaus was later found in~\cite{elizalde_descents_2021}, expressed as a compositional inverse of the generating function~\eqref{eq:EulerianGF}, from where an analogue of Equation~\eqref{eq:Eulerian} for such permutations is deduced. The extension to multisets containing $k$ copies of each element in $[n]$, for any fixed $k$, is studied in~\cite{elizalde_descents_2021} as well. Further generalizations of these results to multisets with an arbitrary number of copies of each element have recently been obtained by Yan et al.~\cite{yan_statistics_2021,yan_partial_2022} and by Fu and Li~\cite{fu_rooted_2022}.

A permutation $\pi$ of $\nn$ can be viewed as a labeled matching of $[2n]$, by placing an arc with label $k$ between $i$ with $j$ if $\pi_i=\pi_j=k$. Let $\mat(\pi)$ denote the matching of $[2n]$ that is obtained by ignoring the labels on the arcs; see the examples in Figure~\ref{fig:matching}. The condition that $\pi$ avoids $1212$ and $2121$ is equivalent to the fact that $\mat(\pi)$ is a {\em noncrossing matching},
meaning that it contains no pair of arcs $(i,\ell)$ and $(j,m)$ such that $i<j<\ell<m$; see 
\cite[Exer.\ 60]{stanley_catalan_2015}. With this perspective, it is natural to also consider permutations
for which $\mat(\pi)$ is a {\em nonnesting matching}, meaning that it contains no pair of arcs $(i,m)$ and $(j,\ell)$ such that $i<j<\ell<m$; see 
\cite[Exer.\ 64]{stanley_catalan_2015}.

\begin{figure}[htb]
\centering
  \begin{tikzpicture}[scale=0.6]
   \foreach [count=\i] \j in {4,4,3,1,1,5,2,2,5,3}
        \node[pnt,label=below:$\j$] at (\i,0)(\i) {};
   \arc{1}{2} \arc{3}{10} \arc{4}{5} \arc{6}{9} \arc{7}{8}
   
  \begin{scope}[shift={(13,0)}]
   \foreach [count=\i] \j in {3,5,3,2,5,2,1,4,1,4}
        \node[pnt,label=below:$\j$] at (\i,0)(\i) {};
   \arc{1}{3} \arc{2}{5} \arc{4}{6} \arc{7}{9} \arc{8}{10}
   \end{scope}
   \end{tikzpicture}
\caption{The matchings corresponding to the permutations $4431152253\in\QQ_5$ (left) and $3532521414\in\C_5$ (right).}
\label{fig:matching}
\end{figure}
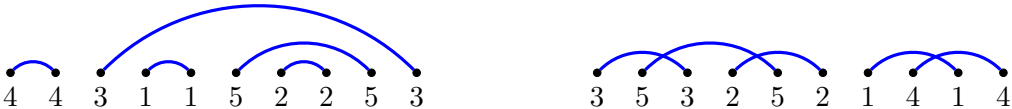

\begin{definition}\label{def:nonnesting}
A permutation $\pi$ of the multiset $\nn$ is called {\em nonnesting} (or a {\em canon permutation}) if it avoids the patterns $1221$ and $2112$; equivalently, if there do not exist $i<j<\ell<m$ such that $\pi_i=\pi_m$ and $\pi_j=\pi_\ell$.
Denote by $\C_n$ the set of nonnesting permutations of $\nn$.
\end{definition}

A permutation $\pi$ of $\nn$ is nonnesting if and only if the subsequence of $\pi$ determined by the first copy of each entry coincides with the subsequence determined by the second copy of each entry. This subsequence, which is a permutation in $\S_n$, will be denoted by $\s(\pi)$.
For example, if $\pi=3532521414\in\C_5$, then $\s(\pi)=35214\in\S_5$.

While the term {\em nonnesting} in Definition~\ref{def:nonnesting} refers to the property of the underlying matching, the alternative term {\em canon permutation} alludes to the musical form where the melody is first played by one voice and then repeated by another voice. This interpretation will play a role in the generalization discussed in Section~\ref{sec:generalization}.

As in the noncrossing case, it is well known \cite[Exer.\ 64]{stanley_catalan_2015} that the number of nonnesting matchings of $[2n]$ is again the $n$th Catalan number; see \cite{chen_crossings_2007} for more information on crossings and nestings in matchings. Since there are $n!$ ways to assign labels to the arcs of a nonnesting matching to form a nonnesting permutation, it follows that
$$
|\C_n|=n!\,\Cat_n=\frac{(2n)!}{(n+1)!}.
$$
In particular, $|\C_n|=|\QQ_n|$. Note however that $\Q_n\not\subseteq\C_n$ in general\footnote{This is stated incorrectly in~\cite[Sec.\ 5]{elizalde_descents_2021}.}; for example, $1221$ is in $\Q_2$ but not in $\C_2$.

Motivated by the results on the distribution of the number of descents and plateaus on Stirling and quasi-Stirling permutations, the goal of this paper is to describe the distribution of these statistics on the set $\C_n$ of nonnesting permutations. We are interested in the polynomial
\begin{equation}\label{eq:polydef}
\poly_n(t,u)=\sum_{\pi\in\C_n} t^{\des(\pi)}u^{\plat(\pi)}.
\end{equation}

Even though the sets $\C_n$ and $\QQ_n$ have the same cardinality, the distribution of descents on them is different. For example, the permutation in $\C_n$ with most descents is $n(n-1)\dots 1n(n-1)\dots 1$, which has $2n-2$ descents, whereas the maximum number of descents of a permutation in $\QQ_n$ is $n-1$, as shown in~\cite{archer_pattern_2019,elizalde_descents_2021}. The distribution of plateaus is also different: $\C_n$ contains permutations with no plateaus, but $\QQ_n$ does not.

\subsection{Dyck paths and Narayana numbers}\label{sec:Narayana}

Let $\D_n$ be the set of lattice paths from $(0,0)$ to $(n,n)$ with steps $\E=(1,0)$ and $\N=(0,1)$ that do not go above the diagonal $y=x$. Elements of $\D_n$ are called Dyck paths. 
There is a standard bijection between nonnesting matchings of $[2n]$ and $\D_n$, obtained by letting the $i$th step of the path be an $\E$ (resp. $\N$) if $i$ is the smaller (resp. larger) endpoint of an arc in the  matching.

A {\em peak} (respectively, a {\em valley}) in a Dyck path is an occurrence of two adjacent steps $\E\N$ (respectively, $\N\E$). A peak is called a {\em low peak} (or a {\em hill}) if these steps touch the diagonal, and it is called a {\em high peak} otherwise\footnote{This terminology comes from the alternative visualization of Dyck paths as paths with steps $(1,1)$ and $(1,-1)$ starting and ending on the $x$-axis and not going below it.}. Denote the number of peaks, the number of low peaks, and the number of high peaks of $D\in\D_n$ by $\pea(D)$, $\lpea(D)$ and $\hpea(D)$, respectively. If $n\ge1$, the number of valleys of $D$ equals $\pea(D)-1$.

Consider the polynomials $$N_n(t,u)=\sum_{D\in\D_n}t^{\hpea(D)}u^{\lpea(D)},$$
and their generating function
$$N(t,u,z)=\sum_{n\ge0}N_n(t,u)z^n.$$
The usual decomposition of Dyck paths by the first return to the diagonal, as $D=\E D_1 \N D_2$, gives the equation
\begin{equation}\label{eq:Ntuz} N(t,u,z)=1+z(N(t,t,z)-1+u)N(t,u,z),\end{equation}
since peaks of $D_1$ become high peaks of $D$, whereas if $D_1$ is empty then $D$ begins with a low peak.
Setting $u=t$ and solving for $N(t,t,z)$, we obtain
$$N(t,t,z)=\sum_{n\ge0}\sum_{D\in\D_n}t^{\pea(D)}z^n=\frac{1+(1-t)z-\sqrt{1-2(1+t)z+(1-t)^2z^2}}{2z}.$$
Now Equation~\eqref{eq:Ntuz} gives 
$$N(t,u,z)=\frac{1}{1-z(N(t,t,z)-1+u)}=\frac{2}{1+(1+t-2u)z+\sqrt{1-2(1+t)z+(1-t)^2z^2}}.$$

It follows from the above expressions that $t(N(t,1,z)-1)=N(t,t,z)-1$. For $n\ge1$, extracting the coefficient of $z^n$ gives the identity 
\begin{equation}\label{eq:Nt1Ntt} tN_n(t,1)=N_n(t,t).\end{equation} 
It is well known that 
\begin{equation}\label{eq:Narayana_numbers} N_n(t,t)=\sum_{r=1}^n \frac{1}{n}\binom{n}{r}\binom{n}{r-1} t^{r}.\end{equation}
The coefficients of this polynomial are called the {\em Narayana numbers}. We sometimes refer to $N_n(t,t)$ and  $N_n(t,u)$ as {\em Narayana polynomials}.

Equation~\eqref{eq:Nt1Ntt} states that the number of paths in $\D_n$ with $r-1$ high peaks equals the number of those with $r$ peaks, for all $r\in[n]$.
A bijective proof of this fact was given by Deutsch~\cite{deutsch_bijection_1998}. Let us describe another simple bijection $\rho:\D_n\to\D_n$ proving this fact. We think of $D\in\D_n$ as a path from the lower-left corner to the upper-right corner of an $n\times n$ grid of unit squares, which we call {\em cells}. 
The steps of each peak $\E\N$ and each valley $\N\E$ of $D$ consist of two sides of a cell; we say that the peak or valley {\em occurs} at that cell.
Any $D\in\D_n$ is determined by the set of cells where its high peaks occur. Let $\rho(D)$ be the unique path in $\D_n$ whose valleys occur precisely at the cells in that set. If $D$ has $r-1$ high peaks, then $\rho(D)$ has $r-1$ valleys, and so it has $r$ peaks. Thus 
\begin{equation}\label{eq:hpea=pea-1}
    \hpea(D)=\pea(\rho(D))-1.
\end{equation} The bijection $\rho$ is essentially the inverse of the {\em rowmotion} map on order ideals of the poset of positive roots in type A; see~\cite{striker_promotion_2012} for details.

An interesting property of the polynomial~\eqref{eq:Nt1Ntt} is that it is palindromic, i.e., 
\begin{equation}\label{eq:Nttpali}
N(t,t)=t^{n+1}N(1/t,1/t).
\end{equation}
This means that the number of paths in $\D_n$ with $r$ peaks equals the number of those with $n+1-r$ peaks, for all $r\in[n]$. While this follows from the expression~\eqref{eq:Narayana_numbers}, next we present a bijective proof, due to Lalanne~\cite{lalanne_involution_1992} and Kreweras~\cite{kreweras_sur_1970}, which will be used later on. 
We use Cartesian coordinates to describe the vertices of the $n\times n$ grid where we draw Dyck paths, so that the lower-left corner is the origin, the lower-right corner is $(n,0)$, and the upper-right corner is $(n,n)$.

Given $D\in\D_n$, suppose that the coordinates of the vertices at the corner of each peak $\E\N$ are $(x_i,y_i)$, where $0<x_1<x_2<\dots<x_r=n$ and $0=y_1<y_2<\dots<y_r<n$.
Let 
$$\{0,1,\dots,n\}\setminus\{x_1,x_2,\dots,x_r\}=\{x'_1,x'_2,\dots,x'_{n+1-r}\}$$
and
$$\{0,1,\dots,n\}\setminus\{y_1,y_2,\dots,y_r\}=\{y'_1,y'_2,\dots,y'_{n+1-r}\},$$
where $0=x'_1<x'_2<\dots<x'_{n+1-r}<n$ and $0<y'_1<y'_2<\dots<y'_{n+1-r}=n$.
Let $\LK(D)$ be the unique path in $\D_n$ such that the coordinates of the corners of its peaks $\E\N$ are $(y'_i,x'_i)$ for $1\le i\le n+1-r$. In particular,
\begin{equation}\label{eq:LKpea}\pea(D)=n+1-\pea(\LK(D)).\end{equation}
The bijection $\LK:\D_n\to\D_n$ is in fact an involution, sometimes referred to as the Lalanne--Kreweras involution.

This bijection has a more convenient visualization if, instead of $\LK(D)$, we draw its reflection along the diagonal $y=x$, which is the path from $(0,0)$ to $(n,n)$ with steps $\N$ and $\E$ such that the coordinates of the corners $\N\E$ are $(x'_i,y'_i)$ for $1\le i\le n+1-r$. See Figure~\ref{fig:LK} for an example.

\begin{figure}
\centering
\begin{tikzpicture}[scale=.6]
\doublegrid{9}   
 \foreach [count=\i] \x in {1,3,4,5,6,9} {
    	\draw[blue,->] (\x,-.5)--(\x,-.2);
    	\draw[blue] (\x,-.5) node[below] {$x_\i$};
 }
 \foreach [count=\i] \y in {0,1,2,3,4,5} {
    	\draw[blue,->] (9.5,\y)--(9.2,\y);
    	\draw[blue] (9.5,\y) node[right] {$y_\i$};
 }
 \draw[very thick,blue](0,0)  \e\n\e\e\n\e\n\e\n\e\n\e\e\e\n\n\n\n;
 \draw[blue] (5.5,3.5) node {$D$};
 \foreach [count=\i] \x in {0,2,7,8} {
    	\draw[magenta,->] (\x,-.5)--(\x,-.2);
    	\draw[magenta] (\x,-.5) node[below] {$x'_\i$};
 }
 \foreach [count=\i] \y in {6,7,8,9} {
    	\draw[magenta,->] (9.5,\y)--(9.2,\y);
    	\draw[magenta] (9.5,\y) node[right] {$y'_\i$};
 }
 \draw[very thick,magenta](0,0)  \n\n\n\n\n\n\e\e\n\e\e\e\e\e\n\e\n\e;
 \draw[magenta] (4.5,7.5) node {$\LK(D)$};
\end{tikzpicture}
\caption{An example of the involution $\LK$, where the path $\LK(D)$ has been reflected along the diagonal for easier visualization.}
\label{fig:LK}
\end{figure}
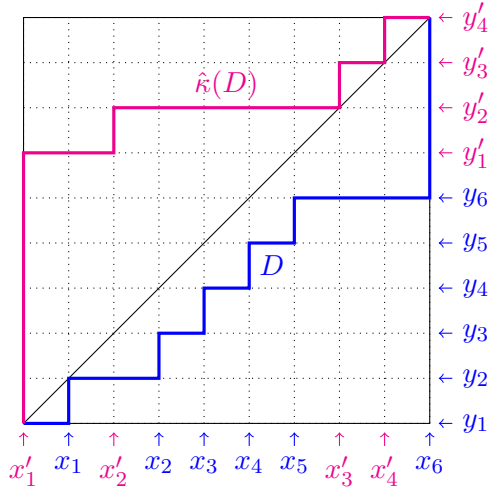

One can check that the composition $\LK\circ\rho$ is also an involution on $\D_n$. In addition to the property $\hpea(D)=n-\pea(\LK(\rho(D)))$ for all $D\in\D_n$, which follows from Equations~\eqref{eq:hpea=pea-1} and~\eqref{eq:LKpea}, it also satisfies that $\lpea(D)=\lpea(\LK(\rho(D)))$, proving the equality 
\begin{equation}\label{eq:Ntupali} 
N_n(t,u)=t^n N_n(1/t,u/t).
\end{equation}
These properties become clear when interpreting the map $D\mapsto\LK(\rho(D))$ in terms of $321$-avoiding permutations, where it is equivalent to the operation sending each permutation to its inverse;
see~\cite{elizalde_fixed_2012,adenbaum_rowmotion_nodate} for more details.

The paper is structured as follows. In Section~\ref{sec:results} we present our main results. In Section~\ref{sec:proofs} we prove the formula giving the distribution of the number of descents and plateaus on nonnesting permutations, by introducing several bijections described in terms of Dyck paths.
In Section~\ref{sec:proofs_symmetry} we use these bijections, along with the Lalanne-Kreweras involution, to prove some symmetry properties of the distributions of the number of descents and weak descents on nonnesting permutations. Finally, Section~\ref{sec:generalization} discusses some generalizations of our results to permutations of the multiset with $k$ copies of each entry, for any fixed $k$.

\section{Main results}\label{sec:results}

Our main result is the following strikingly simple expression for the polynomial $\poly_n(t,u)$, defined in Equation~\eqref{eq:polydef}, as a product of an Eulerian polynomial and a Narayana polynomial.

\begin{theorem}\label{thm:main} 
For $n\ge1$,
$$\poly_n(t,u)=A_n(t)\,N_n(t,u).$$
\end{theorem}

\begin{example}\label{ex:n4} Here are the polynomials $\poly_n(t,u)$ for $n\le 4$, along with the factorizations given by Theorem~\ref{thm:main}:
\begin{align*}
\poly_1(t,u)&=u,\\
\poly_2(t,u)&={u}^{2}+ ( 1+{u}^{2}) t+{t}^{2}=\left(1+t\right)\left({u}^{2}+t\right),\\
\poly_3(t,u)&={u}^{3}+ ( 1+2u+4{u}^{3} ) t+ ( 5+8u+{u}^{3}
 ) {t}^{2}+ ( 5+2u ) {t}^{3}+{t}^{4}\\
 &=\left(1+4t+t^2\right) \left({u}^{3}+ ( 1+2u ) t+{t}^{2}
\right),\\
\poly_4(t,u)&={u}^{4}+ ( 1+2u+3{u}^{2}+11{u}^{4} ) t+ ( 15+24u+33{u}^{2}+11{u}^{4} ) {t}^{2}+ ( 56+44u+33{u}^{2}+{u}^{4} ) {t}^{3}\\ & \quad + ( 56+24u+3{u}^{2} ) {t}^{4}
+ ( 15+2u ) {t}^{5}+{t}^{6}\\
&=\left(1+11t+11t^2+t^3\right) \left({u}^{4}+ ( 1+2u+3{u}^{2} ) t+ ( 4+2u ) {t}^{2}+{t}^{3}\right).
\end{align*}
\end{example}

Theorem~\ref{thm:main} will be proved in Section~\ref{sec:proofs}.
As a consequence of this factorization, we obtain the following two statements about the symmetry of the distribution of descents and weak descents on nonnesting permutations.

\begin{corollary}\label{cor:symmetric_wdes}
The distribution of the number of weak descents on $\C_n$ is symmetric, in the sense that 
$$|\{\pi\in\C_n:\wdes(\pi)=r\}|=|\{\pi\in\C_n:\wdes(\pi)=2n-r\}|$$
for all $r$.
\end{corollary}

\begin{proof}
This follows from Theorem~\ref{thm:main} and the fact that both $A_n(t)$ and $N_n(t,t)$ are palindromic. Indeed, Equations~\eqref{eq:Apali} and~\eqref{eq:Nttpali} imply that $$\poly_n(t,t)=A_n(t)N_n(t,t)=t^{2n}A(1/t)N_n(1/t,1/t)=t^{2n}\poly_n(1/t,1/t).\qedhere$$
\end{proof}

\begin{corollary}\label{cor:symmetric_des}
The distribution of the number of descents on $\C_n$ is symmetric, in the sense that 
$$|\{\pi\in\C_n:\des(\pi)=r\}|=|\{\pi\in\C_n:\des(\pi)=2n-2-r\}|$$
for all $r$.
\end{corollary}

\begin{proof}
In addition to Equation~\eqref{eq:Apali}, we now use the fact that $N_n(t,1)=t^{n-1}N_n(1/t,1)$, which follows from Equations~\eqref{eq:Nt1Ntt} and~\eqref{eq:Nttpali}. We obtain $$\poly_n(t,1)=A_n(t)N_n(t,1)=t^{2n-2}A(1/t)N_n(1/t,1)=t^{2n-2}\poly_n(1/t,1).\qedhere$$
\end{proof}

\begin{example} Setting $u=1$ in Example~\ref{ex:n4}, we obtain the following palindromic expressions:
\begin{align*}
\poly_1(t,1)&=1,\\
\poly_2(t,1)&=1+2t+t^2=(1+t)(1+t),\\
\poly_3(t,1)&=1+7t+14t^2+7t^3+t^4=(1+4t+t^2)(1+3t+t^2),\\
\poly_4(t,1)&=1+17t+83t^2+134t^3+83t^4+17t^5+t^6=(1+11t+11t^2+t^3)(1+6t+6t^2+t^3).
\end{align*}
\end{example}

It is interesting to note that, in contrast to the symmetry of the Eulerian polynomials, there seems to be no obvious explanation for the symmetries described by Corollaries~\ref{cor:symmetric_wdes} and~\ref{cor:symmetric_des}. We remark that the analogous statements for quasi-Stirling permutations do not hold~\cite{elizalde_descents_2021}.
Bijective proofs of Corollaries~\ref{cor:symmetric_wdes} and~\ref{cor:symmetric_des} will be provided in Section~\ref{sec:proofs_symmetry}.

The bivariate polynomials $\poly_n(t,u)$ also satisfy an identity analogous to Equation~\eqref{eq:Ntupali}, namely 
$$\poly_n(t,u)=t^{2n-1} \poly_n(1/t,u/t),$$
but the proof is straightforward in this case. Indeed, the reversal operation $\pi\mapsto\pi^R$ on $\C_n$ satisfies that 
\begin{equation}\label{eq:desRwdes}
 \des(\pi^R)=2n-1-\wdes(\pi)   
\end{equation}
and $\plat(\pi^R)=\plat(\pi)$.

To establish Theorem~\ref{thm:main}, we will prove a slightly stronger statement. For $\sigma\in\S_n$, define 
$$\C_n^\sigma=\{\pi\in\C_n:\s(\pi)=\sigma\},$$
which can be used to partition the set of nonnesting permutations as 
$\C_n=\bigsqcup_{\sigma\in\S_n} \C_n^\sigma$.
Letting 
\begin{equation}\label{eq:polysidef}
\poly_n^\sigma(t,u)=\sum_{\pi\in\C_n^\sigma}t^{\des(\pi)}u^{\plat(\pi)},
\end{equation}
we have the following refinement, which will be proved in Section~\ref{sec:proofs}.

\begin{theorem}\label{thm:refined}
For all $\sigma\in\S_n$, 
$$\poly_n^\sigma(t,u)=t^{\des(\sigma)}N_n(t,u).$$
\end{theorem}

Note that Theorem~\ref{thm:main} follows immediately from Theorem~\ref{thm:refined}, since 
$$\poly_n(t,u)=\sum_{\sigma\in\S_n}\poly_n^\sigma(t,u)=\sum_{\sigma\in\S_n}t^{\des(\sigma)}N_n(t,u)=A_n(t)N_n(t,u).$$

We conclude this section with two statements about the symmetry of the distributions of descents and weak descents on the set of nonnesting permutations having a fixed underlying~$\sigma$. These are simple consequences of Theorem~\ref{thm:refined} and the palindromicity of $N_n(t,t)$ and $N_n(t,1)$, respectively.

\begin{corollary}\label{cor:symmetric_wdes_sigma}
For each $\sigma\in\S_n$, the distribution of the number of weak descents on $\C^\sigma_n$ is symmetric, in the sense that 
$$|\{\pi\in\C_n^\si:\wdes(\pi)-\des(\sigma)=r\}|=|\{\pi\in\C_n^\si:\wdes(\pi)-\des(\sigma)=n+1-r\}|$$
for all $r$.
\end{corollary}

\begin{corollary}\label{cor:symmetric_des_sigma}
For each $\sigma\in\S_n$, the distribution of the number of descents on $\C^\sigma_n$ is symmetric, in the sense that 
$$|\{\pi\in\C_n^\si:\des(\pi)-\des(\sigma)=r\}|=|\{\pi\in\C_n^\si:\des(\pi)-\des(\sigma)=n-1-r\}|$$
for all $r$.
\end{corollary}

Bijective proofs of Corollaries~\ref{cor:symmetric_wdes_sigma} and~\ref{cor:symmetric_des_sigma} will be given in Section~\ref{sec:proofs_symmetry}.

\section{The distribution of descents and plateaus}\label{sec:proofs}

In this section we prove Theorem~\ref{thm:refined}, which in turn implies Theorem~\ref{thm:main}.

A nonnesting permutation $\pi\in\C_n$ is uniquely determined by the underlying nonnesting matching $\mat(\pi)$ and the subsequence of first (equivalently, second) copies $\s(\pi)\in\S_n$. We denote by $\dy(\pi)$ the Dyck path that corresponds to $\mat(\pi)$ via the standard bijection between nonnesting matchings and Dyck paths described in Section~\ref{sec:Narayana}. This path is obtained by reading $\pi$ from left to right and taking an $\E$ step for each first copy of an element, and an $\N$ step for each second copy; see Figure~\ref{fig:visualization}.

The map $\pi\mapsto(\s(\pi),\dy(\pi))$ is a bijection between $\C_n$ and $\S_n\times\D_n$. 
Given $\sigma\in\S_n$ and $D\in\D_n$, we denote by $\pi=\pi(\sigma,D)$ the unique permutation in $\C_n$ such that $\s(\pi)=\sigma$ and $\dy(\pi)=D$. 
Let $\id_n=12\dots n\in\S_n$ denote the identity permutation; we will omit the subscript $n$ when it is clear from the context.
In Section~\ref{sec:stat}, we show that $\poly_n^{\id}(t,u)=N_n(t,u)$, by interpreting $\des(\pi)$ and $\plat(\pi)$ as Dyck path statistics. 

The computation of $\poly_n^{\si}(t,u)$ for arbitrary $\sigma$ will be done in two stages. 
In Section~\ref{sec:Des} we prove that, assuming $n$ is fixed, the polynomial $\poly_n^\sigma(t,u)$ depends only on $\Des(\sigma)$, by showing that swapping any two non-adjacent entries of $\sigma$ with consecutive values does not change this polynomial.
For any $S\subseteq[n-1]$, this allows us to define define $\poly_n^S(t,u)=\poly_n^\sigma(t,u)$, where $\sigma$ is any permutation in $\S_n$ with $\Des(\sigma)=S$, and to reduce the problem to the case where $\sigma$ has a very specific form.  In Section~\ref{sec:removal}, we prove bijectively that if $S'$ is obtained from $S$ by removing its largest element, then $\poly_n^{S}(t,u)=t \poly_n^{S'}(t,u)$, so one can repeatedly apply this fact to remove all the elements from the descent set of $\sigma$, reducing the problem to the case of the identity permutation.
All our proofs are bijective.

\subsection{Descents as Dyck path statistics}\label{sec:stat}

Let $\pi=\pi(\si,D)\in\C_n$, where $\si\in\S_n$ and $D\in\D_n$. We draw $D$ as a path from the lower-left corner to the upper-right corner of an $n\times n$ grid, staying weakly below the diagonal, and we  label the columns of the grid with $\sigma_1,\sigma_2,\dots,\sigma_n$ from left to right, and similarly the rows of the grid from bottom to top.
The permutation $\pi$ can easily be obtained from this representation by following the steps of the path and reading the labels of the columns and rows traversed. See  Figure~\ref{fig:visualization} for an example.

\begin{figure}[htb]
\centering
\begin{tikzpicture}[scale=.6]
  \begin{scope}[shift={(-17,2)}]
   \foreach [count=\i] \j in {2,5,2,5,3,1,6,3,7,4,1,6,7,4}
        \node[pnt,label=below:$\j$] at (\i,0)(\i) {};
   \arc{1}{3} \arc{2}{4} \arc{5}{8} \arc{6}{11} \arc{7}{12} \arc{9}{13} \arc{10}{14}
   \end{scope}
      \grid{7}
      \def\sig{{0,2,5,3,1,6,7,4}}
      \labelgrid{7}
      \draw[very thick,blue](0,0) \e\e\n\n\e\e\e\n\e\e\n\n\n\n;
      \draw[blue] (4.3,1.5) node {$D$};
\end{tikzpicture}
\caption{Representation of the permutation $\pi=25253163741674\in\C_7$ as a matching (left) and as Dyck path $D=\dy(\pi)$ 
in a labeled grid (right), where $\s(\pi)=2531674$.}
\label{fig:visualization}
\end{figure}
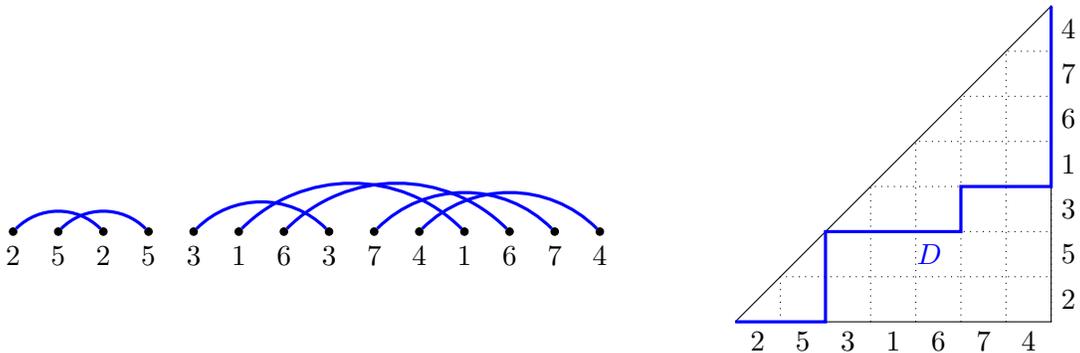

Under this representation, plateaus of $\pi$ correspond to low peaks of $D$, namely, occurrences of adjacent steps $\E\N$ that touch the diagonal, and so 
\begin{equation}\label{eq:plat=lpea}
    \plat(\pi)=\lpea(D).
\end{equation}
Descents of $\pi$ correspond to occurrences in $D$ of any of the following pairs of adjacent steps:
\begin{enumerate}[(i)]
\item $\E\E$ where the labels of the corresponding columns form a descent of $\sigma$,
\item $\N\N$ where the labels of the corresponding rows form a descent of $\sigma$,
\item $\E\N$ where the label of the column is larger than the label of the row,
\item $\N\E$ where the label of the row is larger than the label of the column.
\end{enumerate}

In the special case that $\sigma$ is the identity permutation, all the descents of $\pi$ come from case (iii), and more specifically, from high peaks, namely, occurrences of $\E\N$ that do not touch the diagonal. Thus, 
\begin{equation}\label{eq:des=hpea}
    \des(\pi)=\hpea(D)
\end{equation} in this case, and it follows that 
\begin{equation} \label{eq:identity} 
\poly_n^{\id}(t,u)=\sum_{\pi\in\C_n^{\id}}t^{\des(\pi)}u^{\plat(\pi)}=\sum_{D\in\D_n}t^{\hpea(D)}u^{\lpea(D)}=N_n(t,u).
\end{equation}

At the other extreme, if $\sigma$ is the monotone decreasing permutation $n\ldots21$, the descents of $\pi$ correspond to occurrences of $\E\E$, $\N\N$ and $\N\E$ in $D$. We call these types of occurrences {\em non-peaks}. It follows that $\des(\pi)=2n-1-\pea(D)$, from where
$$\poly_n^{n\ldots21}(t,u)
=\sum_{\pi\in\C^{n\ldots21}_n}t^{\des(\pi)}u^{\plat(\pi)}
=\sum_{D\in\D_n}t^{2n-1-\pea(D)}u^{\lpea(\pi)}
=t^{2n-1} N_n(1/t,u/t)=t^{n-1}N_n(t,u),$$
using Equation~\eqref{eq:Ntupali}.

In general, to easily visualize the descents of $\pi(\sigma,D)$ in the drawing of $D$ on the labeled grid, we will mark the locations on the grid that cause these descents as follows.
For each descent of $\sigma$, say $\sigma_i>\sigma_{i+1}$, we place a red vertical line separating the columns labeled $\sigma_i$ and $\sigma_{i+1}$, and a red horizontal line separating the rows with these labels. Additionally, for every $i,j$, we place a red notch in the upper-left (resp.\ lower-right) corner of the cell in the row with label $\sigma_i$ and the column with label $\sigma_j$ if $\sigma_i>\sigma_j$ (resp.\ $\sigma_i<\sigma_j$). Let $\Gamma_\sigma$ denote the portion of this decorated grid below the diagonal;
see the examples in Figure~\ref{fig:grid}. 

\begin{figure}[htb]
\centering
\begin{tikzpicture}[scale=.6]
      \grid{7}
      \def\sig{{0,2,5,3,1,6,7,4}}
      \labelgridDes{7}
      \decorategrid{7}
\end{tikzpicture}
\hspace{20mm}
 \begin{tikzpicture}[scale=.6]
      \grid{7}
      \def\sig{{0,2,5,3,1,6,7,4}}
      \labelgridDes{7}
      \decorategrid{7}
      \draw[very thick,blue](0,0) \e\e\n\n\e\e\e\n\e\e\n\n\n\n;
      \draw[blue] (4.3,1.5) node {$D$};
\end{tikzpicture}
\caption{The grid $\Gamma_\sigma$ for $\sigma=2531674$ (left), and the Dyck path $D$ from Figure~\ref{fig:visualization} drawn on it, showing that $\des_\sigma(D)=7$ (right).}
\label{fig:grid}
\end{figure}
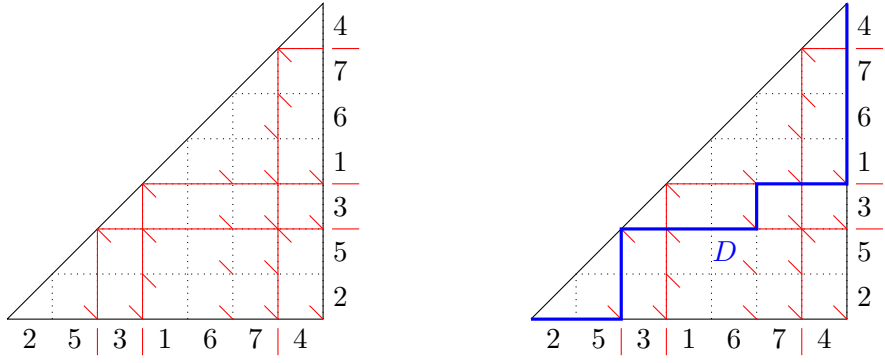

Each descent of $\pi(\sigma,D)$ corresponds to a pair of consecutive steps in $D$ of one of these two types:
\begin{itemize} 
\item a {\em double-step}, defined as a pair $\E\E$ crossing
a red vertical line in $\Gamma_\sigma$, or a pair $\N\N$ crossing a red horizontal line;
\item a {\em corner}, defined as a pair $\E\N$ or $\N\E$ bending around a red notch in $\Gamma_\sigma$. 
\end{itemize}

For each $\sigma\in\S_n$, let $\des_\sigma$ be the Dyck path statistic defined by $$\des_\sigma(D)=\des(\pi(\sigma,D))$$ for $D\in\D_n$.
By construction, $\des_\sigma(D)$ equals the number of double-steps and corners of $D$ in $\Gamma_\sigma$. This defines $n!$ different statistics on $\D_n$, one for each $\sigma\in\S_n$. Theorem~\ref{thm:refined} implies that each statistic $\des_\sigma$ has a (shifted) Narayana distribution, even when refined by the statistic $\lpea$, namely
$$\sum_{D\in\D_n} t^{\des_\sigma(D)}u^{\lpea(D)}=t^{\des(\sigma)} N_n(t,u).$$
The number of high peaks and the number of non-peaks are particular cases of the statistic $\des_\sigma$ when $\sigma=\id_n$ and $\sigma=n\dots 21$, respectively.

\subsection{Switching nonadjacent entries of $\sigma$}\label{sec:Des}

Having verified that Theorem~\ref{thm:refined} holds when $\sigma$ is the identity permutation, the proof for arbitrary $\sigma\in\S_n$ will be obtained by constructing a bijection $$\phi_\sigma:\C^\sigma_n\to\C^{\id}_n$$ satisfying that $\des(\phi_\sigma(\pi))=\des(\pi)-\des(\sigma)$ and $\plat(\phi_\sigma(\pi))=\plat(\pi)$ for all $\pi\in\C_n^\si$. This will be done in two stages, which together transform the underlying permutation $\sigma$ into the identity permutation.

In this subsection, the basic operation that we apply to $\sigma$ is a transposition of two entries with consecutive values $k$ and $k+1$ that are not in adjacent positions. Denote the resulting permutation by $s_k\sigma$, where $s_k=(k,k+1)$ in cycle notation. The assumption that $k$ and $k+1$ are not adjacent in $\sigma$ guarantees that $\Des(s_k\sigma)=\Des(\sigma)$.

\begin{lemma}\label{lem:nonadjacent}
Let $\sigma\in\S_n$ and $k\in[n-1]$, and suppose that the entries $k$ and $k+1$ are not in adjacent positions of $\sigma$.
Then there exists a bijection $f_k:\C_n^\sigma\to\C_n^{s_k\sigma}$ that preserves the pair of statistics $(\des,\plat)$. In particular, $\poly_n^\sigma(t,u)=\poly_n^{s_k\sigma}(t,u)$.
\end{lemma}

\begin{remark}\label{rem:paths}
For $\sigma,\tau\in\S_n$, constructing a bijection $f:\C_n^\sigma\to\C_n^\tau$ that preserves the pair of statistics $(\des,\plat)$ is equivalent to describing a bijection $\hf:\D_n\to\D_n$ such that $\des_{\tau}(\hf(D))=\des_\sigma(D)$ and $\lpea(\hf(D))=\lpea(D)$ for all $D\in\D_n$. Indeed, the desired bijection from $\C_n^\sigma$ to $\C_n^\tau$ is then given by $f(\pi(\sigma,D))=\pi(\tau,\hf(D))$.
\end{remark}

\begin{proof}
We describe a bijection $\hfk:\D_n\to\D_n$ with the properties in Remark~\ref{rem:paths}, where $\tau=s_k\sigma$.
Let $x$ be the cell in $\Gamma_\sigma$ whose row and column are labeled by $\{k,k+1\}$. Note that $x$ is unique, since the cell where the row and column labels are swapped would lie above the diagonal.

Let $D\in\D_n$. If $D$ has steps on two sides of $x$, let $\hfk(D)$ be the path obtained from $D$ by switching these two steps, i.e., from $\N\E$ to $\E\N$ and vice versa;  otherwise, let $\hfk(D)=D$.

To see that $\des_{\tau}(\hfk(D))=\des_\sigma(D)$, note that the only difference between the grids $\Gamma_\sigma$ and $\Gamma_{\tau}$ is that the cell $x$ has a notch in the upper-left corner in one of these grids, and a notch in the lower-right corner in the other. Thus, $\des_{\sigma}$ and $\des_{\tau}$ coincide on all paths except for those that pass through these notches in cell $x$, in which case our construction ensures that $D$ passes through the notch in $\Gamma_\sigma$ if an only if $\hfk(D)$ passes through the notch in $\Gamma_{\tau}$. Additionally, $\lpea(\hfk(D))=\lpea(D)$, since the cell $x$ is not on the diagonal.

Finally, note that the inverse of $\hfk$ is given by $\hfk$ itself. Indeed, since $k$ and $k+1$ are not adjacent in $\sigma$, they are not adjacent in $\tau=s_k\sigma$ either. Additionally, $s_k\tau=\sigma$. The cell in $\Gamma_\tau$ whose row and column are labeled by $\{k,k+1\}$ is the same cell $x$ that had this property in $\Gamma_\sigma$, and so $\hfk(\hfk(D))=D$.
\end{proof}

The bijection on permutations $f_k:\C_n^\sigma\to\C_n^{s_k\sigma}$ that results from the above proof can be described directly as follows: for $\pi\in\C_n^\sigma$, its image $f_k(\pi)\in\C_n^{s_k\sigma}$ is obtained by swapping copies of $k$ with copies of $k+1$ in $\pi$, as long as they are not adjacent to each other, as illustrated in this diagram, where $\rect$ represents a non-empty block of entries, and the brown square illustrates the equivalent changes in cell $x$ of the grid:
$$\begin{array}{c@{\qquad}ccc@{\qquad}c}
&k\rect k{+}1 \rect k \rect k{+}1 & \leftrightarrow & k{+}1\rect k \rect k{+}1 \rect k&\\
&k\rect k \rect k{+}1 \rect k{+}1 & \leftrightarrow & k{+}1\rect k{+}1 \rect k \rect k&\\
   \begin{tikzpicture}[scale=.5]
    \fill[brown!40] (0,0) rectangle (1,1);
    \draw[red] (1,0) -- ++(-.3,.3);
    \draw[thin,dotted] (0,0) rectangle (1,1) ;
    \draw[very thick,blue] (0,0) \e\n;
   \end{tikzpicture}
&k\rect k{+}1\ k \rect k{+}1 & \leftrightarrow & k{+}1\rect k{+}1\  k \rect k&
   \begin{tikzpicture}[scale=.5]
    \fill[brown!40] (0,0) rectangle (1,1);
    \draw[red] (0,1) -- ++(.3,-.3);
    \draw[thin,dotted] (0,0) rectangle (1,1) ;
    \draw[very thick,blue] (0,0) \n\e;
   \end{tikzpicture} \\
   \begin{tikzpicture}[scale=.5] 
    \fill[brown!40] (0,0) rectangle (1,1);
    \draw[red] (1,0) -- ++(-.3,.3);
    \draw[thin,dotted] (0,0) rectangle (1,1) ;
    \draw[very thick,blue] (0,0) \n\e;
   \end{tikzpicture} 
   &k\rect k\  k{+}1 \rect k{+}1 & \leftrightarrow & k{+}1\rect k\ k{+}1 \rect k&
    \begin{tikzpicture}[scale=.5]
    \fill[brown!40] (0,0) rectangle (1,1);
    \draw[red] (0,1) -- ++(.3,-.3);
    \draw[thin,dotted] (0,0) rectangle (1,1) ;
    \draw[very thick,blue] (0,0) \e\n;
   \end{tikzpicture}
\end{array}$$
Note that, since $k$ and $k+1$ are not adjacent in $\sigma$, the first copies of $k$ and $k+1$ in $\pi$ are not adjacent, and neither are the second copies. However, it can happen that the second copy of one of these entries is adjacent to the first copy of the other; in that case, these copies are left unchanged. The permutation $f_k(\pi)$ has the same number of descents and the same number of plateaus as~$\pi$. 

Let $\sim$ be the transitive closure of the binary relation in $\S_n$ that relates $\sigma$ and $s_k\sigma$ whenever entries $k$ and $k+1$ are not adjacent in $\sigma$. Next we show that the equivalence classes of this transitive closure consist of all permutations with a fixed descent set, and we provide a way to pick a canonical representative of each class.

For a permutation $\sigma\in\S_n$, denote its set of {\em non-inversions} by
$$\nI(\sigma)=\{(i,j):1\le i<j\le n,\,\sigma_i<\sigma_j\}.$$
We say that $\sigma$ is {\em reverse-layered} if it can be decomposed as a sequence of monotone increasing blocks where the entries in each block are larger than the entries in the next block. 
For example, $789\ba6\ba345\ba12\in\S_9$ is reverse-layered, where we use vertical bars $\ba$ to indicate the positions of the descents. 
Each block in this decomposition is called a {\em layer}.
For each subset $S=\{d_1,d_2,\dots,d_s\}_<\subseteq[n-1]$, where the subscript $<$ indicates that elements are written in increasing order, denote by 
\begin{equation}\label{eq:la}
    \lambda^S=(n{-}d_1{+}1)(n{-}d_1{+}2)\ldots n\ba(n{-}d_2{+}1)(n{-}d_2{+}2) \ldots (n{-}d_1)\ba\dots\ba12\dots (n{-}d_s)
\end{equation}
the unique reverse-layered permutation in $\S_n$ such that $\Des(\lambda^S)=S$. Note that non-inversions of $\la^S$ occur only between elements in the same layer, and so
$$\card{\nI(\la^S)}=\binom{d_1}{2}+\binom{d_2-d_1}{2}+\dots+\binom{d_s-d_{s-1}}{2}+\binom{n-d_s}{2}.$$
Additionally, any $\sigma\in\S_n$ with $\Des(\sigma)=S$ satisfies $\nI(\la^S)\subseteq \nI(\sigma)$, since any pair of positions with no descent in between must form a non-inversion. 

\begin{lemma}\label{lem:equivalence}
For $\sigma,\tau\in\S_n$, we have $\sigma\sim\tau$ if and only if $\Des(\sigma)=\Des(\tau)$.
\end{lemma}

\begin{proof}
If $k$ and $k+1$ are not adjacent in $\sigma$, then $\Des(\sigma)=\Des(s_k\sigma)$. By definition of the equivalence relation, it follows that if $\sigma\sim\tau$, then $\Des(\sigma)=\Des(\tau)$.

To prove the converse, we will show that if $\Des(\sigma)=S=\{d_1,d_2,\dots,d_s\}_<\subseteq[n-1]$, 
and $\lambda=\lambda^S\in\S_n$ is defined by~\eqref{eq:la}, then $\sigma\sim\lambda$. 

Recall that $\nI(\la)\subseteq \nI(\sigma)$. If $\si\neq\la$, there is some pair $(i,j)\in\nI(\sigma)\setminus\nI(\la)$, and so $i<j$, $\si_i<\si_j$ and $\la_i>\la_j$. In fact, one can always find such a pair with the property that $\si_i$ and $\si_j$ are consecutive, that is, $\si_i=k$ and $\si_j=k+1$ for some $k$. This is because if all pairs $i<j$ with this property were non-inversions of $\la$, then we would have $\si=\la$.
The fact that $k=\si_i<\si_j=k+1$ and $\la_i>\la_j$, together with the assumption $\Des(\si)=\Des(\la)$, guarantee that $k$ and $k+1$ are not adjacent in $\si$, so we have $\sigma\sim s_k\sigma$. Additionally,
$\nI(\la)\subseteq \nI(s_k\sigma)=\nI(\si)\setminus\{(i,j)\}$.
If $s_k\sigma=\la$ we are done, otherwise we repeat the same procedure with $s_k\si$ playing the role of $\si$. Since the number of non-inversions decreases by one at each step, we must eventually reach $\la$.
\end{proof}

A sequence of transpositions as in the above proof, transforming $\sigma$ with $\Des(\si)=S$ into the layered permutation $\la=\la^S$, while preserving the descent set and decreasing the number of non-inversions at each step, will be called a {\em valid sequence}. We can use Lemma~\ref{lem:nonadjacent} to turn any valid sequence into a bijection from $\C_n^\sigma$ to $\C_n^\la$, by composing the bijections $f_k$ in the same order as the transpositions $s_k$ in the sequence. 

A concrete valid sequence can be described as follows. At each step, if $\ell$ is the largest entry that is not in the same position in $\sigma$ and in $\la$, and $i$ is the position such that $\lambda_i=\ell$, we apply the transposition $s_k$ where $k=\sigma_i$. 

Indeed, let us show that $k$ and $k+1$ are not adjacent in $\sigma$, and that $k+1$ appears to the right of $k$, by arguing that $\si_j\neq k+1$ for all $j\le i+1$.
If $\sigma_{i+1}=k+1$, then the fact that $\Des(\si)=\Des(\la)$ would imply that
$\ell=\la_i<\la_{i+1}$, and so $\si_{i+1}=\la_{i+1}$ by the choice of $\ell$; but then $\si_{i+1}>\ell> k$, which is a contradiction.
If $j< i$ and there is no descent separating these two positions, we must have $\si_j< \si_i=k$.
Finally, if $j<i$ and there is a descent separating these positions, then $\la_j$ belongs to a previous layer of $\la$, which implies that $\la_j>\ell>k$, and so $\si_j=\la_j$ by the choice of~$\ell$.
It follows that $k+1$ is to the right of $k$ and these entries are not adjacent, and so one can apply $s_k$ to $\sigma$ while preserving the descent set and decreasing the number of non-inversions.

For example, if $\si=281375496$, we would first apply $s_8$ to get
$s_8\si=291375486$, then apply $s_2$ to get $s_2s_8\si=391275486$, and so on, until we get
\begin{equation}\label{eq:validseq}
    s_1s_3s_2s_4s_3s_2s_1s_5s_4s_3s_2s_6s_7s_6s_5s_4s_3s_2s_8\si=89\ba567\ba4\ba23\ba1.
\end{equation}

Using Lemma~\ref{lem:nonadjacent}, the above construction gives a concrete bijection from $\C_n^\sigma$ to $\C_n^\la$. 
Next we show that, in fact, the choice of the valid sequence is irrelevant, in the sense that the resulting bijection from $\C_n^\sigma$ to $\C_n^\la$ does not depend on this choice.
Let $$\Delta(\si)=\nI(\sigma)\setminus\nI(\la),$$
and let $m=\card{\Delta(\si)}$. Any valid sequence has length $m$, and each transposition $s_k$ in the sequence swaps the entries $k$ and $k+1$ located in positions $i$ and $j$, respectively, for some $(i,j)\in\Delta(\si)$.
The corresponding bijection $\hfk:\D_n\to\D_n$ described in the proof of Lemma~\ref{lem:nonadjacent} switches the steps of the path on the boundary of the cell whose row and column are labeled by $\{k,k+1\}$, which is the cell in the $j$th column and the $i$th row of the grid, where we count columns from left to right and rows from bottom to top. We say that $\hfk$ {\em flips} the cell $(j,i)$. Even though different valid sequences translate into different orderings of the $m$ cells
$\Delta'(\si)=\{(j,i):(i,j)\in\Delta(\si)\}$, let us show that cells that are adjacent (i.e., sharing a side) must always be flipped in the same relative order. 

Indeed, if two cells are adjacent horizontally, say $(j,i),(j+1,i)\in\Delta'(\si)$, then $\si_i<\si_j,\si_{j+1}$ (and so $\la_i>\la_j,\la_{j+1}$). Thus, in the flips corresponding to a valid sequence, cell $(j,i)$ will be flipped before (respectively, after) cell $(j+1,i)$ if and only if $\si_j<\si_{j+1}$ (respectively, $\si_j>\si_{j+1}$). 
Similarly, if two cells are adjacent vertically, say $(j,i),(j,i+1)\in\Delta'(\si)$, then $\si_i,\si_{i+1}<\si_j$
(and so $\la_i,\la_{i+1}>\la_j$). In this case, cell $(j,i)$ will be flipped before (respectively, after) cell $(j,i+1)$ if and only if $\si_i>\si_{i+1}$ (respectively, $\si_i<\si_{i+1}$). In all cases, the order in which adjacent cells are flipped is determined by $\Des(\si)$, not by the particular valid sequence.

Since cell flips commute with each other if the flipped cells are not adjacent, the above argument guarantees that the composition of flips that arises from any valid sequence of transpositions is always the same
map on Dyck paths. We denote this bijection by $\hf_\si:\D_n\to\D_n$, and the corresponding bijection on permutations by $f_\si:\C_n^\si\to\C_n^\la$;  see Figure~\ref{fig:hf} for an example. By Lemma~\ref{lem:nonadjacent}, the following property holds.

\begin{figure}[htb]
\centering
\begin{tikzpicture}[scale=.5]
\begin{scope}[shift={(-11,0)}]  
      \fill[brown!40] (3,0) rectangle  (9,1);
      \fill[brown!40] (7,1) rectangle  (8,2);
      \fill[brown!40] (5,2) rectangle  (9,4);
      \fill[brown!40] (7,4) rectangle  (8,6);
      \fill[brown!40] (8,5) rectangle  (9,7);
      \grid{9}
      \def\sig{{0,2,8,1,3,7,5,4,9,6}}
      \labelgridDes{9}
      \decorategrid{9}
      \ord{8}{2}{1}
      \ord{4}{1}{2}\ord{7}{1}{3}\ord{6}{1}{4}\ord{9}{1}{5}\ord{5}{1}{6}\ord{8}{1}{7}
      \ord{8}{5}{8}
      \ord{7}{4}{9}\ord{6}{4}{10}\ord{9}{4}{11}\ord{8}{4}{12}
      \ord{7}{3}{13}\ord{6}{3}{14}\ord{9}{3}{15}\ord{8}{3}{16}
      \ord{9}{6}{17}\ord{8}{6}{18}
      \ord{9}{7}{19}
\end{scope}
      \fill[brown!40] (3,0) rectangle  (9,1);
      \fill[brown!40] (7,1) rectangle  (8,2);
      \fill[brown!40] (5,2) rectangle  (9,4);
      \fill[brown!40] (7,4) rectangle  (8,6);
      \fill[brown!40] (8,5) rectangle  (9,7);
      \grid{9}
      \def\sig{{0,2,8,1,3,7,5,4,9,6}}
      \labelgridDes{9}
      \decorategrid{9}
      \draw[very thick,blue](0,0) \e\n\e\e\n\e\n\e\e\e\n\n\e\n\n\e\n\n;
      \gle{4}{.5}{$\to$}
      \gle{5}{.5}{$\gets$}
      \gle{6}{.5}{$\gets$}
      \gle{7}{.5}{$\to$}
      \gle{8}{.5}{$\gets$}
      \gle{6}{2.5}{$\gets$}
      \gle{7}{2.5}{$\to$}
      \gle{8}{2.5}{$\gets$}
      \gle{6}{3.5}{$\gets$}
      \gle{7}{3.5}{$\to$}
      \gle{8}{3.5}{$\gets$}
      \gle{8}{5.5}{$\gets$}
      \gle{7.5}{1}{$\downarrow$}
      \gle{7.5}{2}{$\uparrow$}
      \gle{7.5}{3}{$\downarrow$}
      \gle{7.5}{4}{$\downarrow$}
      \gle{7.5}{5}{$\uparrow$}
      \gle{5.5}{3}{$\downarrow$}
      \gle{6.5}{3}{$\downarrow$}
      \gle{7.5}{3}{$\downarrow$}
      \gle{8.5}{3}{$\downarrow$}
      \gle{8.5}{6}{$\uparrow$}
      \draw[->] (10.5,4)-- node[above] {$\hf_\si$} (11.5,4);
      \draw (4.5,-1.6) node {$\pi=228183175437954696$};
\begin{scope}[shift={(11,0)}]      
\fill[brown!40] (3,0) rectangle  (9,1);
      \fill[brown!40] (7,1) rectangle  (8,2);
      \fill[brown!40] (5,2) rectangle  (9,4);
      \fill[brown!40] (7,4) rectangle  (8,6);
      \fill[brown!40] (8,5) rectangle  (9,7);
      \grid{9}
      \def\sig{{0,8,9,5,6,7,4,2,3,1}}
      \labelgridDes{9}
      \decorategrid{9}
      \draw[very thick,blue](0,0) \e\n\e\e\n\e\n\e\n\e\e\e\n\n\e\n\n\n;
      \draw (4.5,-1.6) node {$f_\si(\pi)=889596576423741231$};
\end{scope}
\end{tikzpicture}
\caption{An example of the bijection $\hf_\si$ with $\si=281375496$. The cells in $\Delta'(\si)$ are shaded. In the left grid, they are numbered in the order in which they are flipped according to the valid sequence~\eqref{eq:validseq}; in the middle grid, the arrows indicate the relative order in which adjacent cells are flipped in {\em any} valid sequence, which uniquely determines the map $\hf_\si$. In the right grid, this bijection has been applied to the blue Dyck path $\dy(\pi)$. Note that $\des(f_\si(\pi))=9=\des(\pi)$ and $\plat(f_\si(\pi))=1=\plat(\pi)$.}
\label{fig:hf}
\end{figure}
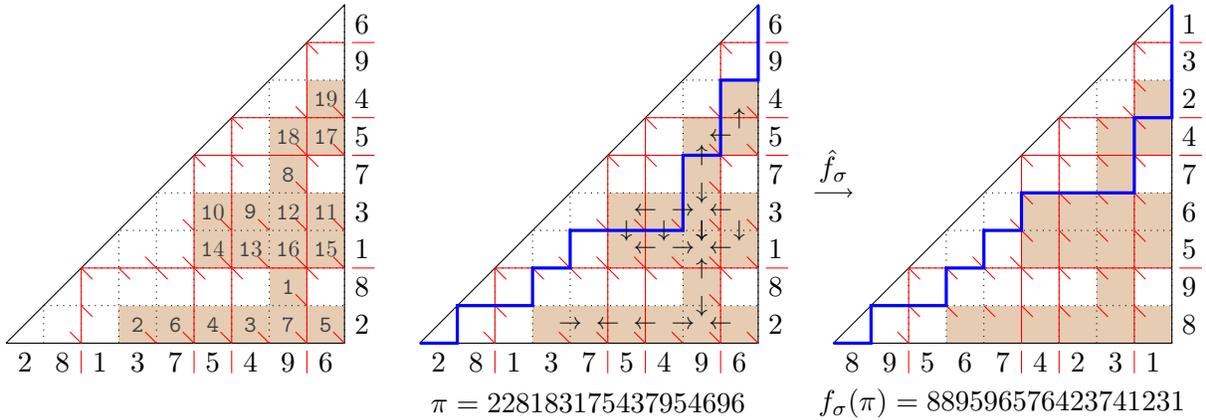

\begin{lemma}
The map $f_\si:\C_n^\si\to\C_n^\la$ defined above is a bijection that preserves the pair of statistics $(\des,\plat)$.
\end{lemma}

As a consequence, the composition $f_\tau^{-1}\circ f_\si$ gives a bijective proof of the following.

\begin{lemma}\label{lem:Des}
If $\sigma,\tau\in\S_n$ are such that $\Des(\sigma)=\Des(\tau)$, then $\poly_n^\sigma(t,u)=\poly_n^\tau(t,u)$.
\end{lemma}

For simplicity of notation, we will write $S$ instead of $\la^S$ when this permutation appears as a subscript or subscript. Specifically, we will write 
$\C_n^S$, $\poly_n^S(t,u)$, $\Gamma_S$ and $\des_S$ instead of $\C_n^{\lambda^S}$, $\poly_n^{\lambda^S}(t,u)$, $\Gamma_{\lambda^S}$ and $\des_{\lambda^S}$, respectively.
Note that $\C_n^\emptyset=\C_n^\id$ by definition.
Lemma~\ref{lem:Des} implies that for any $\sigma\in\S_n$ with $\Des(\sigma)=S$, we have $\poly_n^\sigma(t,u)=\poly_n^{S}(t,u)$. 
\delete{
Some examples of grids $\Gamma_S$ are drawn in Figure~\ref{fig:gridS}.

\begin{figure}[htb]
\centering
\begin{tikzpicture}[scale=.55]
      \grid{9}
      \def\sig{{0,7,8,9,6,3,4,5,1,2}}
      \decorategrid{9}
\end{tikzpicture}
\qquad
 \begin{tikzpicture}[scale=.6]
      \grid{8}
      \def\sig{{0,8,7,5,6,2,3,4,1}}
      \decorategrid{8}
\end{tikzpicture}
\caption{The grids $\Gamma_S$ for $S=\{3,4,7\}\subseteq[9-1]$ and $S=\{1,2,4,7\}\subseteq[8-1]$.}
\label{fig:gridS}
\end{figure}
}

\subsection{Removing descents}\label{sec:removal}

In this subsection we show how the distribution of the number of descents and plateaus on $\C_n^S$ changes when removing elements of $S$ one by one until we reach the empty set, so that $\la^S$ becomes the identity permutation. The key step is the following lemma. 

\begin{lemma}\label{lem:remove}
Let $S\subseteq [n-1]$ be nonempty, let $m=\max S$, and let $S'=S\setminus\{m\}$. Let $\ell=\max S'$ if $S'$ is nonempty, otherwise let $\ell=0$. 
There exists a bijection $g:\C_n^S\to\C_n^{S'}$ such that $\des(g(\pi))=\des(\pi)-1$ and $\plat(g(\pi))=\plat(\pi)$ for all $\pi\in\C_n^S$. In particular, $\poly_n^S(t,u)=t\poly_n^{S'}(t,u)$.
\end{lemma}

\begin{proof}
Similarly to Remark~\ref{rem:paths}, constructing the desired bijection is equivalent to describing a bijection $\hg:\D_n\to\D_n$ such that 
\begin{equation}\label{eq:hgdes}
    \des_{S'}(\hg(D))=\des_{S}(D)-1
\end{equation}
and
\begin{equation}\label{eq:hglpea}
    \lpea(\hg(D))=\lpea(D)
\end{equation}
for all $D\in\D_n$.

Consider the rectangle $R=\{(x,y):\,m\le x\le n,\, 0\le y\le \ell\}$. For $D\in\D_n$, denote by $V$ the portion of $D$ inside $R$, which consists of those steps strictly after the $m$th $\E$ step and before the $(\ell+1)$st $\N$ step of $D$.
If $V$ is empty or it starts and ends with the same type of step ($\E$ or $\N$), let $\hg(D)=D$. 
If $V$ starts with an $\E$ and ends with an $\N$, move the maximal consecutive run of $\E$ steps at the beginning of $V$ to the end of $V$, while keeping the rest of the path the same.
If $V$ starts with an $\N$ and ends with an $\E$, move the maximal consecutive run of $\E$ steps at the end of $V$ to the beginning of $V$, while keeping the rest of the path the same.
Let $\hg(D)$ be the resulting path. See the examples in Figure~\ref{fig:hg}.

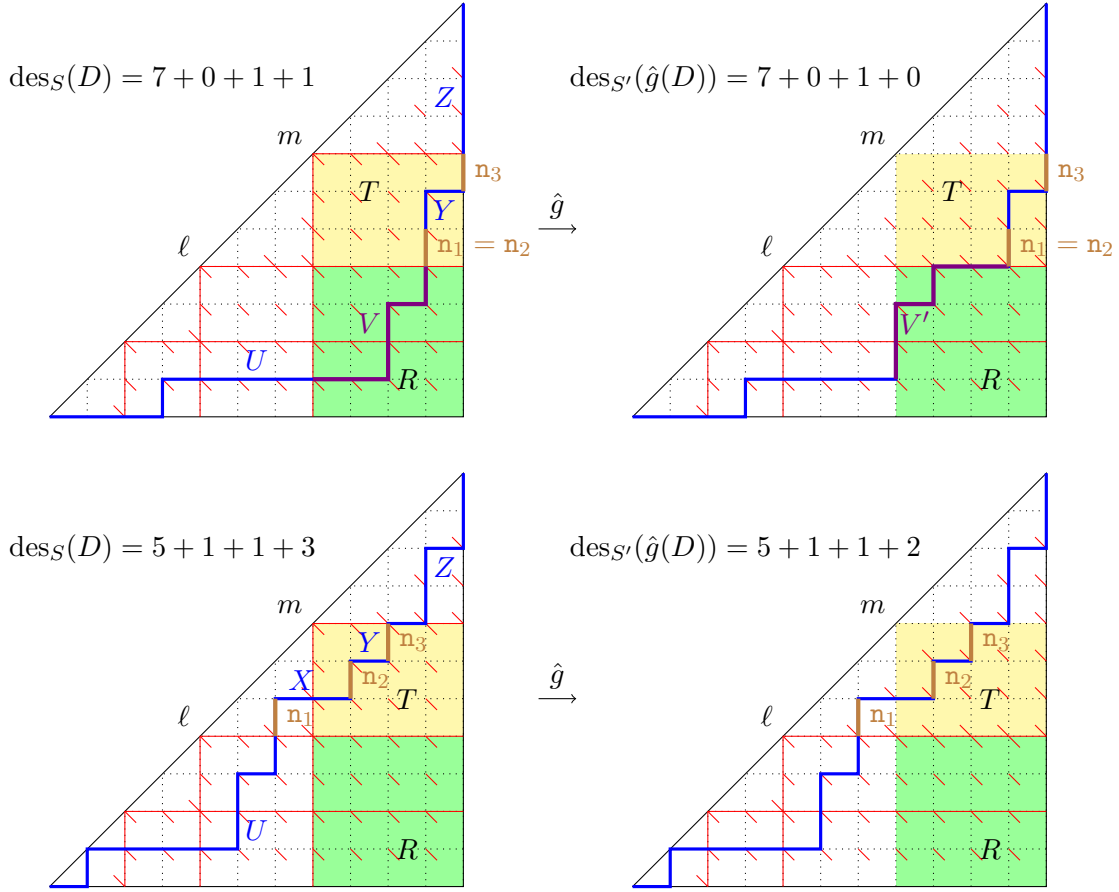
\begin{figure}[htb]
\centering
\begin{tikzpicture}[scale=.5]
      \fill[green!40] (7,0) rectangle  (11,4);
      \fill[yellow!40] (7,4) rectangle  (11,7);
      \draw (9.5,1) node {$R$};
      \draw (8.5,6) node {$T$};      
      \grid{11}
      \def\sig{{0,10,11,8,9,5,6,7,1,2,3,4}}
      \decorategrid{11}
      \draw (4,4) node[above left] {$\ell$};
      \draw (7,7) node[above left] {$m$};
      \draw[very thick,blue](0,0) \e\e\e\n\e\e\e\e \e\e\n\n\e\n \n\n\e\n\n\n\n\n;
      \draw[ultra thick,violet](7,1) \e\e\n\n\e\n;
      \draw[ultra thick,brown] (10,4) -- node[right] {$\N_1=\N_2$} ++(0,1);
      \draw[ultra thick,brown] (11,6) -- node[right] {$\N_3$} ++(0,1);
      \draw[blue] (5.5,1.5) node {$U$};
      \draw[violet] (8.5,2.5) node {$V$};
      \draw[blue] (10.5,5.5) node {$Y$};
      \draw[blue] (10.5,8.5) node {$Z$};
      \draw (3,9) node {$\des_{S}(D)=7+0+1+1$};
      \draw[->] (13,5)-- node[above] {$\hg$} (14,5);
\begin{scope}[shift={(15.5,0)}]
      \fill[green!40] (7,0) rectangle  (11,4);
      \fill[yellow!40] (7,4) rectangle  (11,7);
      \draw (9.5,1) node {$R$}; 
      \draw (8.5,6) node {$T$};      
      \grid{11}
      \def\sig{{0,10,11,8,9,1,2,3,4,5,6,7}}
      \decorategrid{11}
      \draw (4,4) node[above left] {$\ell$};
      \draw (7,7) node[above left] {$m$};
      \draw[very thick,blue](0,0) \e\e\e\n\e\e\e\e \n\n\e\n\e\e \n\n\e\n\n\n\n\n;
      \draw[ultra thick,violet](7,1) \n\n\e\n\e\e;
      \draw[ultra thick,brown] (10,4) -- node[right] {$\N_1=\N_2$} ++(0,1);
      \draw[ultra thick,brown] (11,6) -- node[right] {$\N_3$} ++(0,1);
      \draw[violet] (7.5,2.5) node {$V'$};
      \draw (3,9) node {$\des_{S'}(\hg(D))=7+0+1+0$};
\end{scope}

\begin{scope}[shift={(0,-12.5)}]
      \fill[green!40] (7,0) rectangle  (11,4);
      \fill[yellow!40] (7,4) rectangle  (11,7);
      \draw (9.5,1) node {$R$};  
      \draw (9.5,5) node {$T$};      
      \grid{11}
      \def\sig{{0,10,11,8,9,5,6,7,1,2,3,4}}
      \decorategrid{11}
      \draw (4,4) node[above left] {$\ell$};
      \draw (7,7) node[above left] {$m$};
      \draw[very thick,blue](0,0) \e\n\e\e\e\e\n\n\e\n\n\e\e\n\e\n\e\n\n\e\n\n;
      \draw[ultra thick,brown] (6,4) -- node[right] {$\N_1$} ++(0,1);
      \draw[ultra thick,brown] (8,5) -- node[right] {$\N_2$} ++(0,1);
      \draw[ultra thick,brown] (9,6) -- node[right] {$\N_3$} ++(0,1);
      \draw[blue] (5.5,1.5) node {$U$};
      \draw[blue] (6.7,5.5) node {$X$};
      \draw[blue] (8.5,6.5) node {$Y$};
      \draw[blue] (10.5,8.5) node {$Z$};
      \draw (3,9) node {$\des_{S}(D)=5+1+1+3$};
      \draw[->] (13,5)-- node[above] {$\hg$} (14,5);
\begin{scope}[shift={(15.5,0)}]
      \fill[green!40] (7,0) rectangle  (11,4);
      \fill[yellow!40] (7,4) rectangle  (11,7);
      \draw (9.5,1) node {$R$};  
      \draw (9.5,5) node {$T$};      
      \grid{11}
      \def\sig{{0,10,11,8,9,1,2,3,4,5,6,7}}
      \decorategrid{11}
      \draw (4,4) node[above left] {$\ell$};
      \draw (7,7) node[above left] {$m$};
      \draw[very thick,blue](0,0) \e\n\e\e\e\e\n\n\e\n\n\e\e\n\e\n\e\n\n\e\n\n;
      \draw[ultra thick,brown] (6,4) -- node[right] {$\N_1$} ++(0,1);
      \draw[ultra thick,brown] (8,5) -- node[right] {$\N_2$} ++(0,1);
      \draw[ultra thick,brown] (9,6) -- node[right] {$\N_3$} ++(0,1);
      \draw (3,9) node {$\des_{S'}(\hg(D))=5+1+1+2$};
\end{scope}
\end{scope}
\end{tikzpicture}
\caption{Two examples of the bijection $\hg$, with $n=11$ and $S=\{2,4,7\}$.}
\label{fig:hg}
\end{figure}

It is clear that $\hg$ is a bijection ---in fact, an involution--- from $\D_n$ to itself. Equation~\eqref{eq:hglpea} follows easily since the rectangle $R$ contains no diagonal cells, and hence no low peaks of the paths. It remains to prove Equation~\eqref{eq:hgdes}.

The grids $\Gamma_{S'}$ and $\Gamma_{S}$ differ in only two features: $\Gamma_{S}$ has a red vertical line at $x=m$ and a red horizontal line at $y=m$, and the cells inside the rectangle $T=\{(x,y):\,m\le x\le n,\, \ell\le y\le m\}$ have notches in their lower-right corners in the grid $\Gamma_{S'}$ but in their upper-left corners in the grid $\Gamma_{S}$. 
To compare $\des_{S}(D)$ with $\des_{S'}(\hg(D))$, we will break the paths $D$ and $\hg(D)$ into four factors and compare the contributions of each factor to the statistics $\des_S$ and $\des_{S'}$, respectively. Since the contributions are given by pairs of consecutive steps, namely double-steps and corners, the step at the end of each factor appears again at the beginning of the next factor.

Let $\N_1$ denote the $(\ell+1)$st $\N$ step of $D$, and let $\N_2$ and $\N_3$ denote the first and the last $\N$ steps inside the rectangle $T$, respectively. Note that $\N_3$ is the $m$th $\N$ step of $D$, and that some of these distinguished $\N$ steps may coincide.
These steps induce a decomposition $D=UV\N_1X\N_2Y\N_3Z$, with $V$ defined above, and with the understanding that if $\N_1$, $\N_2$ and $\N_3$ coincide, then $\N_1X\N_2Y\N_3$ is just a single $\N$ step, and similarly if only $\N_1$ and $\N_2$ or only $\N_1$ and $\N_2$ coincide.
We then have $\hg(D)=UV'\N_1X\N_2Y\N_3Z$, where $V'$ is obtained from $V$ by moving the initial or final run of $\E$ steps as described above.
These decompositions give
\begin{align*}\des_S(D)&=\des_S(UV\N_1)+\des_S(\N_1X\N_2)+\des_S(\N_2Y\N_3)+\des_S(\N_3Z),\\ \des_{S'}(\hg(D))&=\des_{S'}(UV'\N_1)+\des_{S'}(\N_1X\N_2)+\des_{S'}(\N_2Y\N_3)+\des_{S'}(\N_3Z).
\end{align*}
We will show that 
\begin{enumerate}[(i)]
\item $\des_S(UV\N_1)=\des_{S'}(UV'\N_1)$, 
\item $\des_S(\N_1X\N_2)=\des_{S'}(\N_1X\N_2)$,  
\item $\des_S(\N_2Y\N_3)=\des_{S'}(\N_2Y\N_3)$,
\item $\des_S(\N_3Z)=\des_{S'}(\N_3Z)+1$,
\end{enumerate}
from where Equation~\eqref{eq:hgdes} will follow.

To prove (i), note that the only two situations where $UV\N_1$ may contribute differently to $\des_S$ and to $\des_{S'}$ are the following. If $V$ starts with an $\E$, then this step and the preceding one (which must be an $\E$ as well) contribute a double-step $\E\E$ to $\des_S$, but not to $\des_{S'}$. If $V$ ends with an $\E$, 
then this step and $\N_1$ contribute a corner $\E\N$ to $\des_{S'}$, but not to $\des_{S}$.

It follows that, if $V$ is empty, or it starts and ends with $\N$, then neither of these situations occurs. 
If $V$ starts and ends with an $\E$, then both situations occur and the differences in contributions cancel each other out, so $\des_S(UV\N_1)=\des_{S'}(UV\N_1)$. In all these cases, $V'=V$ by definition, so (i) holds. 

If $V$ starts with an $\E$ and ends with an $\N$, then $\des_{S'}(UV\N_1)=\des_S(UV\N_1)-1$, since the double-step $\E\E$ crossing the line $x=m$ contributes to $\des_S$ but not to $\des_{S'}$. Let us show that changing $V$ to $V'$ compensates for this by increasing $\des_{S'}$ by~$1$. In this case, the path $V'$ starts with an $\N$ and ends with an $\E$, and it has the same number of double-steps $\N\N$ and one additional $\N\E$ corner than the path $V$. Also, the contribution to $\des_{S'}$ of the double-step $\N\N$ at the end of $UV\N_1$ matches the contribution of the new corner $\E\N$ at the end of $UV'\N_1$; see the top example in Figure~\ref{fig:hg}. All in all, due to the additional $\N\E$ corner, we get that $\des_{S'}(UV'\N_1)=\des_{S'}(UV\N_1)+1=\des_S(UV'\N_1)$.

Similarly, if $V$ starts with an $\N$ and ends with an $\E$, then $\des_{S'}(UV\N_1)=\des_S(UV\N_1)+1$, because of the contribution to $\des_{S'}$ of the corner $\E\N$ at the end of $UV\N_1$. In this case, the path $V'$ starts with an $\E$ and ends with an $\N$, and it has the same number of double-steps $\N\N$ and one fewer $\N\E$ corner than the path $V$. The contribution to $\des_{S'}$ of the corner $\E\N$ at the end of $UV\N_1$ matches the contribution of the new double-step $\N\N$ at the end of $UV\N_1$. All in all, we get that $\des_{S'}(UV'\N_1)=\des_{S'}(UV\N_1)-1=\des_S(UV\N_1)$, proving~(i).

Statement (ii) is clear if $\N_1$ and $\N_2$ coincide, in which case $\N_1X\N_2$ is just a single $\N$ step, or if $\N_2$ is on the left boundary of $T$, in which case $\N_1X$ lies to the left of $T$, where the grids $\Gamma_S$ and $\Gamma_{S'}$ agree. 
If neither of these hold, then $\N_1$ is outside of $T$, and $D$ has no steps on the left boundary of $T$; see the bottom example in Figure~\ref{fig:hg}. In this case, $D$ must cross the left boundary of $T$ via a pair of steps $\E\E$, which contributes a double-step to $\des_S$ but not to $\des_{S'}$, compensating for the contribution of the $\E\N$ corner at the end of $\N_1X\N_2$ to $\des_{S'}$ but not to $\des_S$.

Statement (iii) follows from the fact that $\N_2Y\N_3$ lies entirely inside $T$ and has the same number of $\N\E$ and $\E\N$ corners, which contribute to $\des_S$ and $\des_{S'}$, respectively.

To prove (iv), consider the two possibilities for the step following $\N_3$. If it is an $\N$, it creates a double-step $\N\N$ contributing to $\des_S$ but not to $\des_{S'}$; if it is an $\E$, it creates a corner $\N\E$ which again contributes to $\des_S$ but not to $\des_{S'}$. In the rest of the path $Z$, the grids $\Gamma_S$ and $\Gamma_{S'}$ agree. 
\end{proof}

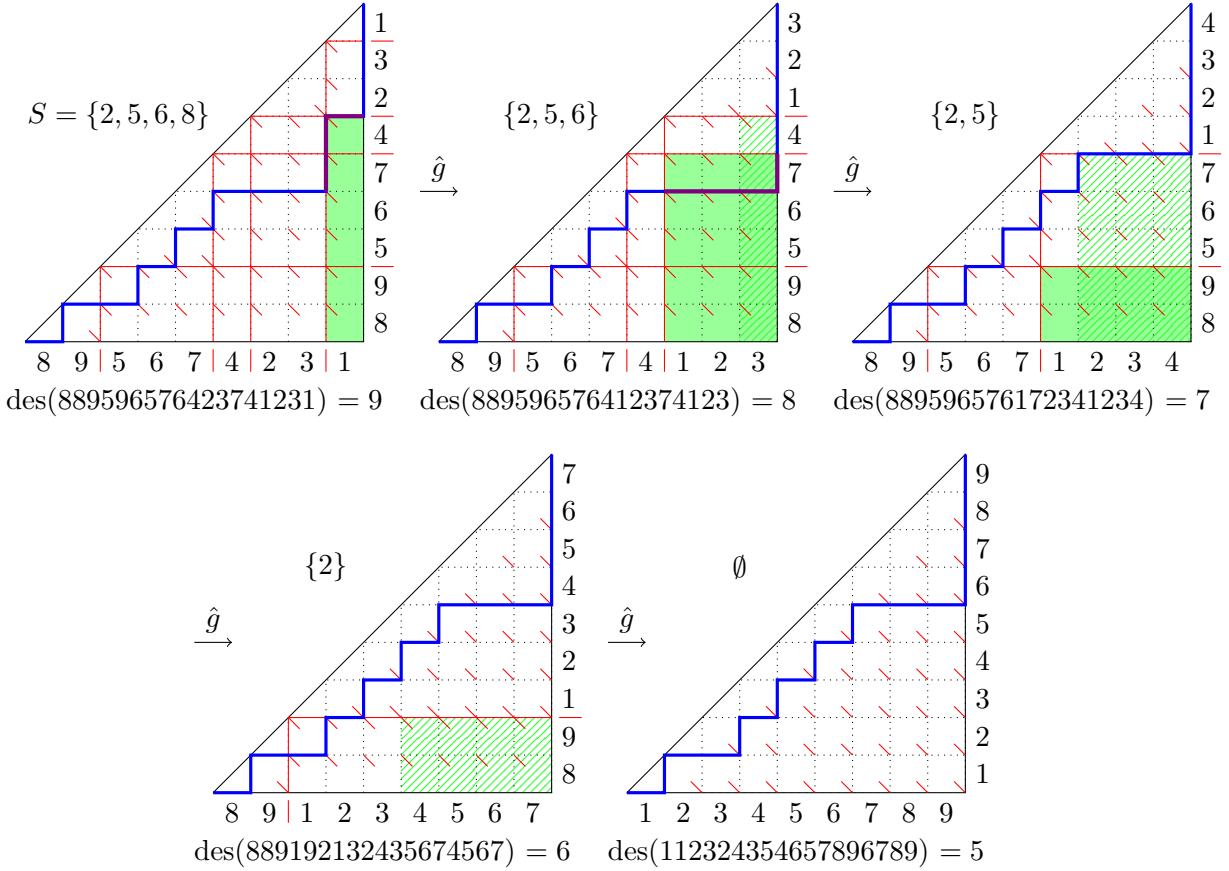
\begin{figure}[htb]
\centering
\begin{tikzpicture}[scale=.5]
      \fill[green!40] (8,0) rectangle  (9,6);
      \grid{9}
      \def\sig{{0,8,9,5,6,7,4,2,3,1}}
      \labelgridDes{9}
      \decorategrid{9}
      \draw[very thick,blue](0,0) \e\n\e\e\n\e\n\e\n\e\e\e \n\n\e \n\n\n;
      \draw[ultra thick,violet](8,4) \n\n\e;
      \draw[->] (10.5,4)-- node[above] {$\hg$} (11.5,4);
      \draw (4.5,-1.6) node {$\des(889596576423741231)=9$};
      \draw (2.5,6) node {$S=\{2,5,6,8\}$};
\begin{scope}[shift={(11,0)}]      
      \fill[green!40] (6,0) rectangle  (9,5);
      \fill[pattern color=green!80,pattern=north east lines] (8,0) rectangle  (9,6);
      \grid{9}
      \def\sig{{0,8,9,5,6,7,4,1,2,3}}
      \labelgridDes{9}
      \decorategrid{9}
      \draw[very thick,blue](0,0) \e\n\e\e\n\e\n\e\n\e \e\e\e\n \n\n\n\n;
      \draw[ultra thick,violet](6,4) \e\e\e\n;
      \draw[->] (10.5,4)-- node[above] {$\hg$} (11.5,4);
      \draw (4.5,-1.6) node {$\des(889596576412374123)=8$};
      \draw (3,6) node {$\{2,5,6\}$};
\end{scope}
\begin{scope}[shift={(22,0)}]  
    \fill[green!40] (5,0) rectangle  (9,2);
      \fill[pattern color=green!80,pattern=north east lines] (6,0) rectangle  (9,5);
      \grid{9}
      \def\sig{{0,8,9,5,6,7,1,2,3,4}}
      \labelgridDes{9}
      \decorategrid{9}
      \draw[very thick,blue](0,0) \e\n\e\e\n\e\n\e\n\e\n\e\e\e\n\n\n\n;
     \draw (4.5,-1.6) node {$\des(889596576172341234)=7$};
     \draw (3,6) node {$\{2,5\}$};
\end{scope}
\begin{scope}[shift={(5,-12)}]      
\begin{scope}[shift={(-11,0)}] \draw[->] (10.5,4)-- node[above] {$\hg$} (11.5,4); \end{scope}
      \fill[green!40] (2,0) rectangle  (9,0);
      \fill[pattern color=green!80,pattern=north east lines] (5,0) rectangle  (9,2);
      \grid{9}
      \def\sig{{0,8,9,1,2,3,4,5,6,7}}
      \labelgridDes{9}
      \decorategrid{9}
      \draw[very thick,blue](0,0)  \e\n\e\e\n\e\n\e\n\e\n\e\e\e\n\n\n\n;
      \draw[->] (10.5,4)-- node[above] {$\hg$} (11.5,4);
     \draw (4.5,-1.6) node {$\des(889192132435674567)=6$};
      \draw (3,6) node {$\{2\}$};
\begin{scope}[shift={(11,0)}]     
      \fill[pattern color=green!80,pattern=north east lines] (2,0) rectangle  (9,0);
      \grid{9}
      \def\sig{{0,1,2,3,4,5,6,7,8,9}}
      \labelgridDes{9}
      \decorategrid{9}
      \draw[very thick,blue](0,0)  \e\n\e\e\n\e\n\e\n\e\n\e\e\e\n\n\n\n;
       \draw (4.5,-1.6) node {$\des(112324354657896789)=5$};
      \draw (3,6) node {$\emptyset$};
\end{scope}
\end{scope}
\end{tikzpicture}
\caption{An example of the bijection $\hg_S$ with $S=\{2,5,6,8\}$, applied to the path on the right of Figure~\ref{fig:hf}. We obtain $g_S(889596576423741231)=112324354657896789$.}
\label{fig:hgS}
\end{figure}

When $S$ consists of a single element, the bijection $g$ in Lemma~\ref{lem:remove} is just the identity. Indeed, since $\ell=0$ in this case, the rectangle $R$ is the above proof contains no cells, and so $\hg$ does not change the path.

For a given $S\subseteq[n-1]$, denote by $$g_S:\C_n^S\to\C_n^\emptyset$$ the bijection obtained by repeatedly applying the map $g$ from Lemma~\ref{lem:remove} until all the elements of $S$ have been removed, which happens after $|S|$ iterations; see Figure~\ref{fig:hgS} for an example. It follows from the lemma that
$\des(g_S(\pi))=\des(\pi)-|S|$ and $\plat(g_S(\pi))=\plat(\pi)$ for all $\pi\in\C_n^S$. 

For $\sigma\in\S_n$ with $\Des(\sigma)=S$, consider the composition
$$\C_n^\sigma\stackrel{f_\si}{\longrightarrow}\C_n^S\stackrel{g_S}{\longrightarrow}\C_n^\emptyset=\C_n^{\id},$$
and denote by $\phi_\si=g_S\circ f_\si$ the resulting bijection from $\C_n^\sigma$ to $\C_n^{\id}$. The corresponding bijections on Dyck paths are denoted by $\hg_S$ (obtained by repeatedly applying $\hg$) and by $\hphi_\si=\hg_S\circ \hf_\si$.

\begin{proof}[Proof of Theorem \ref{thm:refined}]
Let $S=\Des(\si)$. The bijection $\phi_\sigma:\C_n^\sigma\to\C_n^{\id}$ has the property that, for all $\pi\in\C_n^\si$, 
\begin{equation}\label{eq:desphi}
    \des(\phi_\si(\pi))= \des(g_S(f_\si(\pi)))=\des(f_\si(\pi))-|S|=\des(\pi)-\des(\si)
\end{equation}
and
\begin{equation}\label{eq:platphi}
\plat(\phi_\si(\pi))=\plat(g_S(f_\si(\pi)))=\plat(f_\si(\pi))=\plat(\pi).
\end{equation}
It follows that
$$\poly_n^\si(t,u)=t^{\des(\si)}\poly_n^{\id}(t,u)=t^{\des(\si)}N_n(t,u),$$
by Equation~\eqref{eq:identity}.
\end{proof}

For example, composing the bijections in Figures~\ref{fig:hf} and~\ref{fig:hgS}, we see that if $\si=281375496$ and $\pi=228183175437954696\in\C_9^\si$, then $\phi_\si(\pi)=112324354657896789\in\C_9^\id$. Note that $\des(\phi_\si(\pi))=5=9-4=\des(\pi)-\des(\si)$ and $\plat(\phi_\si(\pi))=1=\plat(\pi)$.

\delete{
\begin{proof}[Proof of Theorem \ref{thm:main}]
Theorem~\ref{thm:refined} states that $\poly_n^\sigma(t,u)=t^{\des(\sigma)}N_n(t,u)$ for all $\sigma\in\S_n$.
Since $\C_n=\bigsqcup_{\sigma\in\S_n}\C^\sigma_n$, it follows that
$$\poly_n(t,u)=\sum_{\sigma\in\S_n} \poly^\sigma_n(t,u)=\sum_{\sigma\in\S_n}t^{\des(\sigma)}N_n(t,u)=A_n(t)N_n(t,u).\qedhere$$
\end{proof}
}

\section{Bijective proofs of symmetry}\label{sec:proofs_symmetry}

In this section we combine the bijection $\phi_\si$ from Section~\ref{sec:proofs} with the involution $\LK$ from Section~\ref{sec:Narayana} to give bijective proofs of Corollaries \ref{cor:symmetric_wdes}, \ref{cor:symmetric_des}, \ref{cor:symmetric_wdes_sigma} and \ref{cor:symmetric_des_sigma}.

One can easily interpret the Lalanne--Kreweras involution $\LK$ as an involution on permutations $\LKC:\C_n^\id\to\C_n^\id$, by identifying each $\pi\in\C_n^\id$ with its associated Dyck path $\dy(\pi)\in\D_n$. Specifically, given $\pi\in\C_n^\id$ with associated Dyck path $D=\dy(\pi)$, define $\LKC(\pi)\in\C_n^\id$ to be the permutation whose associated Dyck path is $\LK(D)$.

If $\pi=\pi(\id,D)\in\C_n^\id$, Equations~\eqref{eq:plat=lpea} and~\eqref{eq:des=hpea} imply that $$\wdes(\pi)=\des(\pi)+\plat(\pi)=\hpea(D)+\lpea(D)=\pea(D).$$
Thus, the behavior of $\LK$ on the number of peaks described in Equation~\eqref{eq:LKpea} translates to the fact that, for any $\pi\in\C_n^\id$,
\begin{equation}\label{eq:LKCwdes}
\wdes(\pi)=n+1-\wdes(\LKC(\pi)).
\end{equation}

\begin{proof}[Bijective proof of Corollary~\ref{cor:symmetric_wdes_sigma}]
Define an involution $\Phi_\si=\phi_\si^{-1}\circ\LKC\circ\phi_\si$ from $\C_n^\si$ to itself, as illustrated by the following diagram:

{\centering
\begin{tikzpicture}[scale=1.2]
     \draw(0,1) node {$\C_n^\si$};
     \draw[->] (.4,1)--node[above] {$\phi_\si$} (1.1,1);
     \draw(1.5,1) node {$\C_n^\id$};
     \draw[<->] (0,.7)--node[left] {$\Phi_\sigma$} (0,.3);
     \draw(0,0) node {$\C_n^\si$};
     \draw[->] (.4,0)--node[above] {$\phi_\si$} (1.1,0);
     \draw(1.5,0) node {$\C_n^\id$};
     \draw[<->] (1.5,.7)--node[right] {$\LKC$} (1.5,.3);
\end{tikzpicture}
\par}

Adding Equations~\eqref{eq:desphi} and~\eqref{eq:platphi}, we get
\begin{equation}\label{eq:wdesphi}
    \wdes(\phi_\si(\pi))=\wdes(\pi)-\des(\si)
\end{equation}
for every $\pi\in\C_n^\sigma$. Using this property twice and applying Equation~\eqref{eq:LKCwdes}, it follows that
\begin{equation}\label{eq:wdesPhi}
    \wdes(\pi)-\des(\si)=\wdes(\phi_\si(\pi))=n+1-\wdes(\LKC(\phi_\si(\pi)))=n+1-(\wdes(\Phi_\si(\pi))-\des(\si))
\end{equation}
for every $\pi\in\C_n^\sigma$. Thus, the map $\Phi_\si$ provides the desired bijection. See Figure~\ref{fig:Phi} for an example.
\end{proof}

\begin{figure}[htb]
\centering
\begin{tikzpicture}[scale=.48]
      \fill[brown!40] (3,0) rectangle  (9,1);
      \fill[brown!40] (7,1) rectangle  (8,2);
      \fill[brown!40] (5,2) rectangle  (9,4);
      \fill[brown!40] (7,4) rectangle  (8,6);
      \fill[brown!40] (8,5) rectangle  (9,7);
      \fill[brown!40] (0,3) rectangle  (1,9);
      \fill[brown!40] (1,7) rectangle  (2,8);
      \fill[brown!40] (2,5) rectangle  (4,9);
      \fill[brown!40] (4,7) rectangle  (6,8);
      \fill[brown!40] (5,8) rectangle  (7,9);
      \doublegrid{9}
      \def\sig{{0,2,8,1,3,7,5,4,9,6}}
      \labelgridDes{9}
      \draw[very thick,blue](0,0) \e\n\e\e\n\e\n\e\e\e\n\n\e\n\n\e\n\n;
      \gle{4}{.5}{$\to$}
      \gle{5}{.5}{$\gets$}
      \gle{6}{.5}{$\gets$}
      \gle{7}{.5}{$\to$}
      \gle{8}{.5}{$\gets$}
      \gle{6}{2.5}{$\gets$}
      \gle{7}{2.5}{$\to$}
      \gle{8}{2.5}{$\gets$}
      \gle{6}{3.5}{$\gets$}
      \gle{7}{3.5}{$\to$}
      \gle{8}{3.5}{$\gets$}
      \gle{8}{5.5}{$\gets$}
      \gle{7.5}{1}{$\downarrow$}
      \gle{7.5}{2}{$\uparrow$}
      \gle{7.5}{3}{$\downarrow$}
      \gle{7.5}{4}{$\downarrow$}
      \gle{7.5}{5}{$\uparrow$}
      \gle{5.5}{3}{$\downarrow$}
      \gle{6.5}{3}{$\downarrow$}
      \gle{7.5}{3}{$\downarrow$}
      \gle{8.5}{3}{$\downarrow$}
      \gle{8.5}{6}{$\uparrow$}
      \draw[very thick,magenta](0,0) \n\n\n\n\n\e\n\e\e\e\n\e\e\e\n\e\n\e;
      \gle{.5}{4}{$\uparrow$}
      \gle{.5}{5}{$\downarrow$}
      \gle{.5}{6}{$\downarrow$}
      \gle{.5}{7}{$\uparrow$}
      \gle{.5}{8}{$\downarrow$}
      \gle{2.5}{6}{$\downarrow$}
      \gle{2.5}{7}{$\uparrow$}
      \gle{2.5}{8}{$\downarrow$}
      \gle{3.5}{6}{$\downarrow$}
      \gle{3.5}{7}{$\uparrow$}
      \gle{3.5}{8}{$\downarrow$}
      \gle{5.5}{8}{$\downarrow$}
      \gle{1}{7.5}{$\gets$}
      \gle{2}{7.5}{$\to$}
      \gle{3}{7.5}{$\gets$}
      \gle{4}{7.5}{$\gets$}
      \gle{5}{7.5}{$\to$}
      \gle{3}{5.5}{$\gets$}
      \gle{3}{6.5}{$\gets$}
      \gle{3}{7.5}{$\gets$}
      \gle{3}{8.5}{$\gets$}
      \gle{6}{8.5}{$\to$}
      \Desdots{{2,5,6,8}}
      \draw[blue] (3.7,1.5) node {$D$};
      \draw[magenta] (4,4.5) node {$\hphi_\si^{-1}(\LK(\hphi_\si(D)))$};
      \draw[->] (10.25,4)-- node[above] {$\hf_\si$} (11.25,4);
\begin{scope}[shift={(11.5,0)}]      
      \fill[green!40] (8,0) rectangle  (9,6);
      \fill[green!40] (0,8) rectangle  (6,9);
      \fill[pattern color=brown!80,pattern=north east lines] (3,0) rectangle  (9,1);
      \fill[pattern color=brown!80,pattern=north east lines] (7,1) rectangle  (8,2);
      \fill[pattern color=brown!80,pattern=north east lines] (5,2) rectangle  (9,4);
      \fill[pattern color=brown!80,pattern=north east lines] (7,4) rectangle  (8,6);
      \fill[pattern color=brown!80,pattern=north east lines] (8,5) rectangle  (9,7);
      \fill[pattern color=brown!80,pattern=north east lines] (0,3) rectangle  (1,9);
      \fill[pattern color=brown!80,pattern=north east lines] (1,7) rectangle  (2,8);
      \fill[pattern color=brown!80,pattern=north east lines] (2,5) rectangle  (4,9);
      \fill[pattern color=brown!80,pattern=north east lines] (4,7) rectangle  (6,8);
      \fill[pattern color=brown!80,pattern=north east lines] (5,8) rectangle  (7,9);
      \doublegrid{9}
      \gle{4}{.5}{$\to$}
      \gle{5}{.5}{$\gets$}
      \gle{6}{.5}{$\gets$}
      \gle{7}{.5}{$\to$}
      \gle{8}{.5}{$\gets$}
      \gle{6}{2.5}{$\gets$}
      \gle{7}{2.5}{$\to$}
      \gle{8}{2.5}{$\gets$}
      \gle{6}{3.5}{$\gets$}
      \gle{7}{3.5}{$\to$}
      \gle{8}{3.5}{$\gets$}
      \gle{8}{5.5}{$\gets$}
      \gle{7.5}{1}{$\downarrow$}
      \gle{7.5}{2}{$\uparrow$}
      \gle{7.5}{3}{$\downarrow$}
      \gle{7.5}{4}{$\downarrow$}
      \gle{7.5}{5}{$\uparrow$}
      \gle{5.5}{3}{$\downarrow$}
      \gle{6.5}{3}{$\downarrow$}
      \gle{7.5}{3}{$\downarrow$}
      \gle{8.5}{3}{$\downarrow$}
      \gle{8.5}{6}{$\uparrow$}
       \gle{.5}{4}{$\uparrow$}
      \gle{.5}{5}{$\downarrow$}
      \gle{.5}{6}{$\downarrow$}
      \gle{.5}{7}{$\uparrow$}
      \gle{.5}{8}{$\downarrow$}
      \gle{2.5}{6}{$\downarrow$}
      \gle{2.5}{7}{$\uparrow$}
      \gle{2.5}{8}{$\downarrow$}
      \gle{3.5}{6}{$\downarrow$}
      \gle{3.5}{7}{$\uparrow$}
      \gle{3.5}{8}{$\downarrow$}
      \gle{5.5}{8}{$\downarrow$}
      \gle{1}{7.5}{$\gets$}
      \gle{2}{7.5}{$\to$}
      \gle{3}{7.5}{$\gets$}
      \gle{4}{7.5}{$\gets$}
      \gle{5}{7.5}{$\to$}
      \gle{3}{5.5}{$\gets$}
      \gle{3}{6.5}{$\gets$}
      \gle{3}{7.5}{$\gets$}
      \gle{3}{8.5}{$\gets$}
      \gle{6}{8.5}{$\to$}
      \def\sig{{0,8,9,5,6,7,4,2,3,1}}
      \labelgridDes{9}
      \draw[very thick,blue](0,0) \e\n\e\e\n\e\n\e\n\e\e\e\n\n\e\n\n\n;
      \draw[ultra thick,violet](8,4) \n\n\e;
     \draw[very thick,magenta](0,0) \n\n\n\n\n\n\e\e\n\e\e\e\e\e\n\e\n\e;
      \Desdots{{2,5,6,8}}
      \draw[->] (10.25,4)-- node[above] {$\hg$} (11.25,4);
\end{scope}
\begin{scope}[shift={(23,0)}]   
      \fill[green!40] (6,0) rectangle  (9,5);
      \fill[pattern color=green!80,pattern=north east lines] (8,0) rectangle  (9,6);
      \fill[green!40] (0,6) rectangle  (5,9);
      \fill[pattern color=green!80,pattern=north east lines] (0,8) rectangle  (6,9);
      \doublegrid{9}
      \def\sig{{0,8,9,5,6,7,4,1,2,3}}
      \labelgridDes{9}
      \draw[very thick,blue](0,0) \e\n\e\e\n\e\n\e\n\e \e\e\e\n \n\n\n\n;
      \draw[ultra thick,violet](6,4) \e\e\e\n;
      \draw[very thick,magenta](0,0) \n\n\n\n\n\n\e\e\n\e\e\e\e\e\n\e\n\e;
      \draw[ultra thick,violet](0,6) \e\e\n\e\e\e;
      \Desdots{{2,5,6}}
\end{scope}
\begin{scope}[shift={(1,-11.5)}]    
\begin{scope}[shift={(-11.5,0)}] \draw[->] (10.25,4)-- node[above] {$\hg$} (11.25,4); \end{scope}
      \fill[green!40] (5,0) rectangle  (9,2);
      \fill[pattern color=green!80,pattern=north east lines] (6,0) rectangle  (9,5);
      \fill[green!40] (0,5) rectangle  (2,9);
      \fill[pattern color=green!80,pattern=north east lines] (0,6) rectangle  (5,9);
      \doublegrid{9}
      \def\sig{{0,8,9,5,6,7,1,2,3,4}}
      \labelgridDes{9}
      \draw[very thick,blue](0,0) \e\n\e\e\n\e\n\e\n\e\n\e\e\e\n\n\n\n;
      \draw[very thick,magenta](0,0) \n\n\n\n\n\n\e\e\n\e\e\e\e\e\n\e\n\e;
      \draw[ultra thick,violet](0,5) \n\e\e\n;
      \Desdots{{2,5}}
      \draw[->] (10.25,4)-- node[above] {$\hg$} (11.25,4);
\begin{scope}[shift={(11.5,0)}]    
      \fill[green!40] (2,0) rectangle  (9,0);
      \fill[pattern color=green!80,pattern=north east lines] (5,0) rectangle  (9,2);
      \fill[green!40] (0,2) rectangle  (0,9);
      \fill[pattern color=green!80,pattern=north east lines] (0,5) rectangle  (2,9);
      \doublegrid{9}
      \def\sig{{0,8,9,1,2,3,4,5,6,7}}
      \labelgridDes{9}
      \draw[very thick,blue](0,0)  \e\n\e\e\n\e\n\e\n\e\n\e\e\e\n\n\n\n;
      \draw[very thick,magenta](0,0) \n\n\n\n\n\n\e\e\n\e\e\e\e\e\n\e\n\e;
      \Desdots{{2}}
      \draw[->] (10.25,4)-- node[above] {$\hg$} (11.25,4);
\end{scope}
\begin{scope}[shift={(23,0)}]
    \fill[pattern color=green!80,pattern=north east lines] (2,0) rectangle  (9,0);
    \fill[pattern color=green!80,pattern=north east lines] (0,2) rectangle  (0,9);
      \doublegrid{9}
      \def\sig{{0,1,2,3,4,5,6,7,8,9}}
      \labelgridDes{9}
      \draw[very thick,blue](0,0)  \e\n\e\e\n\e\n\e\n\e\n\e\e\e\n\n\n\n;
      \draw[very thick,magenta](0,0) \n\n\n\n\n\n\e\e\n\e\e\e\e\e\n\e\n\e;
       \foreach [count=\i] \x in {1,3,4,5,6,9} {
    	\draw[blue,->] (\x,-.5)--(\x,-.2); }
      \foreach [count=\i] \y in {0,1,2,3,4,5} {
    	\draw[blue,->] (9.5,\y)--(9.2,\y);}
    \foreach [count=\i] \x in {0,2,7,8} {
    	\draw[magenta,->] (\x,-.5)--(\x,-.2);}
    \foreach [count=\i] \y in {6,7,8,9} {
    	\draw[magenta,->] (9.5,\y)--(9.2,\y);}
      \draw[<->,thick] (5.5,6.5)--node[above right] {$\LK$} (6.5,5.5);
       \draw[blue] (5.4,2.4) node {$\hphi_\si(D)$};
       \draw[magenta] (4,7.6) node {$\LK(\hphi_\si(D))$};
\end{scope}
\end{scope}
\end{tikzpicture}
\caption{
An example of the bijection $\Phi_\si$ with $\si=281375496$, having $\Des(\si)=S=\{2,5,6,8\}$, applied to $\pi=228183175437954696\in\C_9^\si$. First the map $\phi_\si=g_S\circ f_\si$ is computed by applying $\hphi_\si=\hg_S\circ\hf_\si$ to the blue Dyck path $D=\dy(\pi)$ on the upper-left grid, as in Figures~\ref{fig:hf} and~\ref{fig:hgS}. Then $\LK$ is applied to the resulting path to obtain the magenta path in the bottom-right grid, which is drawn reflected above the diagonal for convenience. Finally, $\hphi_\si^{-1}=\hf_\si^{-1}\circ \hg_S^{-1}$ is applied to this path, to obtain the magenta path in the upper-left grid, giving $\Phi_\si(\pi)=281372581347549966$. Note that $\wdes(\pi)=10$ and $\wdes(\Phi_\si(\pi))=8$, satisfying Equation~\eqref{eq:wdesPhi}.
}
\label{fig:Phi}
\end{figure}

\begin{proof}[Bijective proof of Corollary~\ref{cor:symmetric_wdes}]
Recall that $\si^R$ denotes the reversal of $\si$, and consider now the composition $\Psi_\si=\phi_{\si^R}^{-1}\circ\LKC\circ\phi_\si$, which is a bijection from $\C_n^\si$ to $\C_n^{\si^R}$, as illustrated by the diagram:

{\centering
\begin{tikzpicture}[scale=1.2]
     \draw(0,1) node {$\C_n^\si$};
     \draw[->] (.4,1)--node[above] {$\phi_\si$} (1.1,1);
     \draw(1.5,1) node {$\C_n^\id$};
     \draw[->] (0,.7)--node[left] {$\Psi_\sigma$} (0,.3);
     \draw(0,0) node {$\C_n^{\si^R}$};
     \draw[->] (.4,0)--node[above] {$\phi_{\si^R}$} (1.1,0);
     \draw(1.5,0) node {$\C_n^\id$};
     \draw[<->] (1.5,.7)--node[right] {$\LKC$} (1.5,.3);
\end{tikzpicture}
\par}

By definition, the inverse of $\Psi_\si$ is simply $\Psi_{\si^R}$. Equation~\eqref{eq:wdesphi} for $\sigma$ and $\sigma^R$, together with Equation~\eqref{eq:LKCwdes}, imply that
$$
\wdes(\pi)-\des(\si)=\wdes(\phi_\si(\pi))=n+1-\wdes(\LKC(\phi_\si(\pi)))=n+1-(\wdes(\Psi_\si(\pi))-\des(\si^R)).
$$
Using Equation~\eqref{eq:desRsi}, it follows that
\begin{equation}\label{eq:wdesPsi}
\wdes(\pi)=2n-\wdes(\Psi_\si(\pi)).
\end{equation}
Thus, the involution $\Psi:\C_n\to\C_n$ defined by $\Psi(\pi)=\Psi_\si(\pi)$ whenever $\pi\in\C_n^\sigma$ provides the desired bijection.
For example, for $\pi$ as in Figure~\ref{fig:Phi}, one can compute that $\Psi(\pi)=694573691457318822$.
\end{proof}

\begin{proof}[Bijective proof of Corollary~\ref{cor:symmetric_des_sigma}] 
Denote by $R$ the reversal map that takes $\pi$ to $\pi^R$, and define the composition $\ol\Phi_\si=R\circ\Phi_{\si^R}\circ R$, which is an involution from $\C_n^\si$ to itself:

{\centering
\begin{tikzpicture}[scale=1.2]
     \draw(0,1) node {$\C_n^\si$};
     \draw[->] (.4,1)--node[above] {$R$} (1.1,1);
     \draw(1.5,1) node {$\C_n^{\si^R}$};
     \draw[<->] (0,.7)--node[left] {$\ol\Phi_\si$} (0,.3);
     \draw(0,0) node {$\C_n^\si$};
     \draw[->] (.4,0)--node[above] {$R$} (1.1,0);
     \draw(1.5,0) node {$\C_n^{\si^R}$};
     \draw[<->] (1.5,.7)--node[right] {$\Phi_{\si^R}$} (1.5,.3);
\end{tikzpicture}
\par}

Let $\pi\in\C_n^\sigma$. By Equations~\eqref{eq:desRwdes} and~\eqref{eq:desRsi},
\begin{equation}\label{eq:desR}
    \des(\pi)-\des(\si)=2n-1-\wdes(\pi^R)-(n-1-\des(\si^R))=n-(\wdes(\pi^R)-\des(\si^R)).
\end{equation}
On the other hand, Equation~\eqref{eq:wdesPhi} with $\pi^R$ and $\si^R$ playing the roles of $\pi$ and $\sigma$, respectively, states that
$$\wdes(\pi^R)-\des(\si^R)=n+1-(\wdes(\Phi_{\si^R}(\pi^R))-\des(\si^R)).$$
Applying Equation~\eqref{eq:desR} twice, it follows that
\begin{align*}
\des(\pi)-\des(\si)&=n-(\wdes(\pi^R)-\des(\si^R))=-1+\wdes(\Phi_{\si^R}(\pi^R))-\des(\si^R)\\
&=n-1-(\des(\ol\Phi_\si(\pi))-\des(\si)),
\end{align*}
and so $\ol\Phi_\si$ gives the desired bijection.
As an example, for $\pi$ and $\si$ as in Figure~\ref{fig:Phi}, one can compute that $\ol\Phi_\si(\pi)=282811337574956496$.
\end{proof}

\begin{proof}[Bijective proof of Corollary~\ref{cor:symmetric_des}]
Consider the composition $\ol\Psi_\si=R\circ\Psi_{\si^R}\circ R$,  which is a bijection from $\C_n^\si$ to $\C_n^{\si^R}$:

{\centering
\begin{tikzpicture}[scale=1.2]
     \draw(0,1) node {$\C_n^\si$};
     \draw[->] (.4,1)--node[above] {$R$} (1.1,1);
     \draw(1.5,1) node {$\C_n^{\si^R}$};
     \draw[->] (0,.7)--node[left] {$\ol\Psi_\si$} (0,.3);
     \draw(0,0) node {$\C_n^{\si^R}$};
     \draw[->] (.4,0)--node[above] {$R$} (1.1,0);
     \draw(1.5,0) node {$\C_n^{\si}$};
     \draw[->] (1.5,.7)--node[right] {$\Psi_{\si^R}$} (1.5,.3);
\end{tikzpicture}
\par}

The inverse of $\ol\Psi_\si$ is $\ol\Psi_{\si^R}$. Using Equation~\eqref{eq:wdesPsi} with $\pi^R$ and $\si^R$ playing the roles of $\pi$ and $\sigma$, respectively, together with Equation~\eqref{eq:desRwdes} applied twice, we get
$$\des(\pi)=2n-1-\wdes(\pi^R)=-1+\wdes(\Psi_{\si^R}(\pi^R))
=2n-2-\des(\ol\Psi_\si(\pi)).$$
Thus, the involution $\ol\Psi:\C_n\to\C_n$ defined by $\ol\Psi(\pi)=\ol\Psi_\si(\pi)$ whenever $\pi\in\C_n^\sigma$ provides the desired bijection.
As an example, for $\pi$ as in Figure~\ref{fig:Phi}, we have $\ol\Psi(\pi)=696944557371823182$.
\end{proof}

It might be possible to find simpler bijections
proving our results about symmetry of the distributions.

\begin{problem}
Give direct bijections proving Corollaries \ref{cor:symmetric_wdes}, \ref{cor:symmetric_des}, \ref{cor:symmetric_wdes_sigma} and \ref{cor:symmetric_des_sigma} that do not require passing through the case where $\sigma$ is the identity permutation.
\end{problem}

\section{Generalizations}\label{sec:generalization}

Generalizations of Stirling and quasi-Stirling permutations to the multiset containing $k$ copies of each number in $[n]$ have been studied in the literature. We denote this multiset by $\nnk$, so that $\bigsqcup^2[n]=\nn$. 

Generalized Stirling permutations, often called $k$-Stirling permutations, are permutations of $\nnk$ that avoid the pattern $212$. This generalization, originally proposed by Gessel and Stanley~\cite{gessel_stirling_1978}, has been studied in~\cite{park_r-multipermutations_1994,park_inverse_1994,janson_generalized_2011,kuba_analysis_2011}. 

Similarly, $k$-quasi-Stirling permutations were defined by the author~\cite{elizalde_descents_2021} as those permutations of $\nnk$ that avoid the patterns $1212$ and $2121$. Permutations of $\nnk$ can be viewed as ordered set partitions on $[kn]$ into $n$ blocks of size $k$, where block $b$ consists of those $i\in[kn]$ such that $\pi_i=b$, for each $b\in[n]$. With this interpretation, a permutation avoids $1212$ and $2121$ if the underlying set partition is {\em noncrossing}, meaning that there are no $i<j<\ell<m$ so that $i,\ell$ are in one block and $j,m$ are in another block; see~\cite{kreweras_sur_1972,edelman_chain_1980} or \cite[Exer.\ 159]{stanley_catalan_2015}.
The distribution of the number of descents and plateaus in $k$-quasi-Stirling permutations is given in~\cite{elizalde_descents_2021}.

There are several ways to generalize nonnesting permutations of $\nn$, as given by Definition~\ref{def:nonnesting}, to permutations of $\nnk$. Next we describe three different generalizations, which arise from distinct ways to view elements of $\C_n$: as pattern-avoiding multipermutations, as labeled nonnesting matchings, and as multipermutations where the subsequences of first copies and of second copies of each entry coincide.

\subsection{Permutations avoiding $1221$ and $2112$}

One possible generalization is to consider permutations $\pi$ of $\nnk$ that avoid the patterns $1221$ and $2112$, that is, there do not exist $i<j<\ell<m$ such that $\pi_i=\pi_m\neq\pi_j=\pi_\ell$. Let $\A_n^{k}$ denote this set of permutations. For example, $\A_2^{3}=\{111222,112122,221211,222111\}$. 
When $k\ge3$, this definition is quite restrictive. The distribution of descents and plateaus is relatively simple in this case, as the next result shows.

\begin{theorem}\label{thm:Ank}
For $k\ge3$ and $n\ge1$, we have
$$\sum_{\pi\in\A_n^{k}} t^{\des(\pi)}u^{\plat(\pi)}=u^{(k-3)n+2}(u^2+t)^{n-1}A_n(t).$$
\end{theorem}

\begin{proof}
Suppose that $\pi\in\A_n^{k}$, and call $1221$ and $2112$ the {\em forbidden} patterns.
We claim that if there exist $i<j<\ell<m$ such that 
$\pi_i\pi_j\pi_\ell\pi_m=abab$ for some $a\neq b$ (this configuration is called a crossing), then $\pi_\ell$ must be the last $a$ in $\pi$, $\pi_j$ must be the first $b$ in $\pi$, and $\ell=j+1$.

Indeed, if $\pi_\ell$ was not the last $a$, then $\pi_r=a$ for some $r$ satisfying either $\ell<r<m$, in which case $\pi_j\pi_\ell\pi_r\pi_m=baab$, or
$m<r$, in which case $\pi_i\pi_j\pi_m\pi_r=abba$, both forbidden patterns. A symmetric argument shows that $\pi_j$ must be the first $b$. 

To show that $\ell=j+1$, suppose for contradiction that there exists an index $r$ with $j<r<\ell$. Having $\pi_r\in\{a,b\}$ would create a forbidden pattern, so we must have $\pi_r=c$ for some $c\notin\{a,b\}$. Consider the possible locations of the other copies of $c$ in $\pi$. If $\pi_s=c$ for some $i<s<m$, then the copies of $c$ in positions $r$ and $s$ would create a forbidden pattern with the copies of $a$ or $b$, so the only possible locations for copies of $c$ are outside of the interval $[i,m]$. If there are copies of $c$ both to the left and to the right of this interval, again they would create a forbidden pattern, so all the copies must be on one side. Since $k\ge3$, $\pi$ must contain at least two copies of $c$ in addition to $\pi_r$. Let us assume by symmetry that these copies are to the right of this interval, say $\pi_s=\pi_t=c$ for $m<s<t$, so that $\pi_i\pi_j\pi_r\pi_\ell\pi_m\pi_s\pi_t=abcabcc$. Now $\pi$ must contain a third copy of $b$, but any possible location for it would create a forbidden pattern. This proves the above claim.

Together with the fact that $\pi$ avoids $1221$ and $2112$, the claim implies that if the first and the last copy of an entry $a$ in $\pi$ are in positions $i$ and $j$, for some $i<j$, then all the entries in between these two must be equal to $a$, with two possible exceptions: $\pi_{i+1}$ may be the last copy of a different entry, and $\pi_{j+1}$ may be the first copy of a different entry.
Thus, for any two elements in $[n]$, all the copies of one of them (say $a$) must be to the left of all the copies of the other (say $b$), with the only possible exception that the last copy of $a$ may be to the left of the first copy of $b$. When this exception occurs, we will say $a$ {\em overruns} $b$.

It follows that, for any given $\pi\in\A_n^k$, there is some permutation $\si\in\S_n$ such that, for 
every $i\in[n-1]$, either all the copies of $\si_i$ in $\pi$ appear to the left of all the copies of $\si_{i+1}$, or $\si_i$ overruns $\si_{i+1}$. Additionally, each of the $n!$ choices for $\si$, combined with each of the $2^{n-1}$ possible choices of indices $i\in[n-1]$ such that $\si_i$ overruns $\si_{i+1}$, produce a (unique) permutation $\pi\in\A_n^k$. This shows that $|\A_n^k|=2^{n-1}n!$.

The $n!$ permutations in $\A_n^k$ with no overruns contribute $u^{(k-1)n}A_n(t)$ to the generating function. For each $i\in[n-1]$, swapping the last copy of $\si_i$ with the first copy of $\si_{i+1}$ to create an overrun decreases $\plat$ by $2$ and increases $\des$ by $1$, independently of the values $\sigma_i,\sigma_{i+1}$ and of the choices made for the other indices. 
It follows that 
$$\sum_{\pi\in\A_n^{k}} t^{\des(\pi)}u^{\plat(\pi)}=u^{(k-1)n}A_n(t)\left(1+\frac{t}{u^2}\right)^{n-1},$$
which is equivalent to the stated formula.
\end{proof}

Setting $u=1$ and $u=t$ in Theorem~\ref{thm:Ank}, respectively, we obtain the expressions
\begin{align*}
\sum_{\pi\in\A_n^{k}} t^{\des(\pi)}&=(1+t)^{n-1}A_n(t),\\
\sum_{\pi\in\A_n^{k}} t^{\wdes(\pi)}&=t^{(k-2)n+1}(1+t)^{n-1}A_n(t)
\end{align*}
for $k\ge3$. It follows that the distribution of the number of descents (respectively, weak descents) on $\A_n^{k}$ is symmetric. A bijective proof of these symmetries is obtained easily by reversing $\sigma$ and toggling the choices of the indices $i\in[n-1]$ where the overruns happen in the proof of Theorem~\ref{thm:Ank}.

\subsection{Nonnesting permutations}

A second way to generalize Definition~\ref{def:nonnesting} to permutations of $\nnk$ is to require that the underlying set partition is {\em nonnesting}; see \cite{athanasiadis_noncrossing_1998,athanasiadis_piles_1999,klazar_identities_2006,kasraoui_distribution_2006,chen_crossings_2007} and \cite[Exer.\ 164]{stanley_catalan_2015}.
In this case, a block $\{i_1,i_2,\dots,i_k\}_<$ of the partition is represented by $k-1$ arcs $(i_1,i_2),(i_2,i_3),\dots,(i_{k-1},i_k)$,
and the partition is nonnesting if this representation does not contain any pair of nested arcs, i.e., two arcs $(i,m)$ and $(j,\ell)$ such that $i<j<\ell<m$. In terms of the permutation $\pi$ of $\nnk$, an arc is placed between $i$ and $j$ if $\pi_i=\pi_j$ and there is no other copy of this value between positions $i$ and $j$ of $\pi$. This is equivalent to requiring that $\pi$ avoids the {\em barred} patterns $12\overline{1}21$ and $21\overline{2}12$; see~\cite{west_sorting_1993,bousquet-melou_forest-like_2007} for other appearances of barred patterns in the literature. 
Following Athanasiadis~\cite{athanasiadis_piles_1999}, who first considered such permutations in a geometric context , we call these {\em nonnesting permutations}.
Let $\B_n^{k}$ denote the set of nonnesting permutations of $\nnk$.
Clearly, $\A_n^{k}\subseteq \B_n^{k}$, but the converse is not true in general. For example, $\B_2^{3}=\A_n^{k}\sqcup\{121212,212121\}$.

A bijection between noncrossing and nonnesting partitions that preserves the sizes of the blocks was given by Athanasiadis~\cite{athanasiadis_noncrossing_1998}. It follows that the number of nonnesting permutations of $\nnk$ equals the number of those where the underlying partition is noncrossing, namely, $k$-quasi-Stirling permutations. This number is given in~\cite{elizalde_descents_2021}, and so we get $$|\B_n^k|=\frac{(kn)!}{((k-1)n+1)!}=n!\,\Cat_{n,k},$$
where $$\Cat_{n,k}=\frac{1}{(k-1)n+1}\binom{kn}{n}$$ is called a {\em $k$-Catalan number} \cite[pp.\ 168--173]{stanley_enumerative_1999}.
Given the simple formula in Theorem~\ref{thm:main} for the $k=2$ case and the results in~\cite{elizalde_descents_2021} on $k$-quasi-Stirling permutations, it is natural to ask the following.
\begin{problem} Find the distribution of the number of descents and the number of plateaus on $\B_n^{k}$ for $k\ge3$.
\end{problem}
It is interesting to point out that, unlike in the case $k=2$, the distribution of the number of descents on $\B_n^{k}$ fails to be symmetric for $k=3$ and $n=4$.

\subsection{Canon permutations}

A third possible generalization arises when thinking of $\C^\si_n$ as the set of permutations of $\nn$ 
where both the subsequence of first copies of each entry 
and the subsequence of second copies of each entry equal $\si$.
For $\si\in\S_n$ and $k\ge1$, define $\C^{k,\sigma}_n$ to be the set of permutations $\pi$ of $\nnk$ such that, for each $j\in[k]$, the subsequence of $\pi$ formed by the $j$th copy (from the left) of each entry is $\sigma$. 
Let $\C^{k}_n=\bigsqcup_{\sigma\in\S_n}\C^{k,\sigma}_n$. 
We call elements of $\C^{k}_n$ {\em $k$-canon permutations}.
By construction, $\A^{1}_n=\B^{1}_n=\C^{1}_n=\S_n$ and $\A^{2}_n=\B^{2}_n=\C^{2}_n=\C_n$, so the three definitions generalize the $k=1$ and $k=2$ cases.

In general, we have that $\B^{k}_n\subseteq\C^{k}_n$. To see this, suppose that $\pi$ is a permutation of $\nnk$ and that $\pi\notin\C^k_n$. Then there must exist some $j\in[k-1]$ so that the subsequences of $\pi$ formed by the $j$th copy and by the $(j+1)$st copy of each entry, respectively, are not equal. Thus, there exist $a,b\in[n]$ such that the $j$th copy of $a$ appears before the $j$th copy of $b$, but the $(j+1)$st copy of $a$ appears after the $(j+1)$st copy of $b$. But then, in the arc representation of the underlying partition, the arc connecting these two copies of $b$ is nested inside the arc connecting these two copies of $a$, and so $\pi\notin\B^k_n$.
 The reverse inclusion does not hold in general. For example, $\C_2^{3}=\B_2^{3}\sqcup\{112212,121122,212211,221121\}$.

In a subsequent paper, we will show that this definition provides the right setting to generalize Theorems~\ref{thm:main} and~\ref{thm:refined}.
We will show that, if we define
$$
\poly^k_n(t,u)=\sum_{\pi\in\C^k_n} t^{\des(\pi)}u^{\plat(\pi)} \quad \text{and} \quad 
\poly^{k,\si}_n(t,u)=\sum_{\pi\in\C^{k,\si}_n} t^{\des(\pi)}u^{\plat(\pi)}
$$
in analogy to Equations~\eqref{eq:polydef} and~\eqref{eq:polysidef}, then
for all $\sigma\in\S_n$ and all $k\ge1$ we have
\begin{equation}\label{eq:conj_refined}\poly_n^{k,\sigma}(t,u)=t^{\des(\sigma)}\poly_n^{k,\id}(t,u).\end{equation}
Thus, summing over all $\sigma\in\S_n$,
$$\poly_n^{k}(t,u)=A_n(t) \poly_n^{k,\id}(t,u).$$

The proof extends some of the ideas in our proofs of Theorems~\ref{thm:main} and~\ref{thm:refined} from Dyck paths to the more general setting of standard Young tableaux of rectangular shape. In particular, it yields certain generalizations of the notion of descents on such tableaux, as well as bijective proofs of some equidistribution results for these new statistics.

\subsection*{Acknowledgements}
The author thanks Kassie Archer for interesting discussions in the early stages of this work.

\bibliographystyle{plain}
\bibliography{nonnesting_multiperm-fixed}

\end{document}